\documentclass[11pt]{article}
\usepackage{amsmath,amssymb,epsf}
\usepackage{amsmath}
\usepackage[applemac]{inputenc}
\advance \textwidth 3cm
\advance \textheight 3cm
\usepackage{amsfonts}
\usepackage{latexsym}
\newtheorem{theorem}{Theorem}
\newtheorem{lemma}{Lemma}
\newtheorem{corollary}{Corollary}
\newtheorem{definition}{D\'efinition}
\newtheorem{proposition}{Proposition}
\newtheorem{prop}{Property}
\newtheorem{remark}{Remark}
\newenvironment{proof}[1]{\par\noindent\underline{Proof #1}:\quad}%
{\unskip\nobreak\hfil\penalty50\hskip2em\null\nobreak\hfil%
$\Box$\parfillskip0pt\par\medskip}
\title{Asymptotic of the terms of the Gegenbauer polynomials on the unit circle and applications to the inverse of Toeplitz matrices..}

\author{ Philippe Rambour\thanks{Universit\'{e} de Paris Sud,
      B\^atiment 425; F-91405
Orsay Cedex;
tel : 01 69 15 57 28 ; fax 01 69 15 60 19
      \mbox{e-mail : philippe.rambour@math.u-psud.fr}
     }}
\date{}
\begin{document}
\maketitle
  \renewcommand{\abstractname}{Abstract}
     \begin{abstract}
     \textbf{Asymptotic of the terms of the Gegenbauer polynomials on the unit circle and applications to the inverse of Toeplitz matrices.}\\
     The first part of this paper is devoted to the study of the orthogonal polynomials on the unit circle, with  
     respect of a weight of type $ f _{\alpha} : \theta \mapsto 2^{2\alpha} (\cos \theta - \cos \theta_{0}) ^{2\alpha} c_{1}$ with 
     $\theta_{0}\in ]0,\pi[$, $-\frac{1}{2} < \alpha<\frac{1}{2}$  and $c_{1}$ a sufficiently smooth function. 
     In a second part of the paper we obtain an asymptotic of the entries 
     $\left(T_{N}f_{\alpha}\right)^{-1} _{k+1,l+1}$ 
     for $\alpha>0$ and for sufficiently large values of $k,l$, with $k\neq l$.
     \end{abstract}
           
\textbf{Secondary} 33D45, 42C05, 42C10.

\section{Introduction}
The study of the orthogonal polynomials on the unit circle is an old and difficult problem (see \cite{BS1}, \cite{BS2} or \cite{SZEG}). 
The Gegenbauer polynomials on the torus are the orthogonal polynomials on the circle with respect 
to a weight of type $f_{\alpha}: \theta\mapsto 2^{2\alpha} (\cos \theta -\cos\theta_{0})^{2\alpha} c_{1}$ with $\alpha>-\frac{1}{2}$ and $c_{1}$ a  positive integrable function. In this paper we assume $-\frac{1}{2}<\alpha\le\frac{1}{2}$ and $c_{1}$  sufficiently smooth regular function. It is said that a function $k$ is regular if 
$k(\theta)>0$ for all $\theta\in \mathbb T$ and $k \in L^1(\mathbb T)$.
In a first part we are interested in  the asymptotic of the coefficients of these polynomials (see Corollary 
\ref {GEGEN}). The main tool to compute this is the study 
of the Toeplitz matrix with symbol $f$. Given a function $h$ in $L^1 (\mathbb T)$ we denote by 
$T_{N}(h)$ the Toeplitz matrix of order $N$ with symbol $h$ the $(N+1)\times (N+1)$ matrix such that 
$$ \left( T_{N}(h)\right) _{i+1,j+1} = \hat h(j-i) \quad \forall i,j \quad 0\le i,j \le N $$
where $\hat m (s)$ is the Fourier coefficient of order $s$ of the function $m$ (see, for instance \cite {Bo.4} and \cite {Bo.5}). There is a close connection between Toeplitz matrices and orthogonal polynomials on the complex unit circle. Indeed the coefficients of the orthogonal polynomial of degree $N$ with respect of $h$ are also the coefficients of the last column 
of $T_N^{-1} (h)$ except for a normalisation (see \cite{Ld}). 
Here we give an asymptotic expansion of the entries  
$\left(T_{N} \left(  f_{\alpha}\right) \right)_{k+1,1}^{-1}$ (Theorem \ref{COEF}). Using 
the symmetries of the Toeplitz matrix $T_{N}(f_{\alpha})$, we deduce from this last result an asymptotic of $\left(T_{N} (f_{\alpha}) \right)_{N-k+1,N+1}^{-1}$ (corollary \ref{GEGEN}).\\
The proof of Theorem \ref{COEF} often refers to results of \cite{RS10}. 
In this last work we have treated the case of the symbols $h_{\alpha}$defined by 
$ \theta \mapsto (1-\cos \theta)^\alpha c$ whith $-\frac{1}{2} < \alpha \leq \frac{1}{2}$ and  
the same hypothesis on  $c$ as on $c_{1}$. 
We have stated the following Theorem which is an important tool in the demonstration of Theorem \ref{COEF}.
\begin{theorem}[\cite{RS10}] \label{PREDIZERO}
If $-\frac{1}{2} <\alpha \le \frac{1}{2}$, $\alpha\neq0$
we have for $c\in A(\mathbb T ,\frac{3}{2})$ and 
$0<x<1$ 
$$ c(1) \left( T_N ( h_{\alpha})\right)^{-1} _{[Nx]+1,1} =  N^{\alpha-1} \frac{1}{\Gamma
(\alpha)}x^{\alpha-1} (1-x)^\alpha + o(N^{\alpha-1}).
$$
uniformly in $x$ for $x\in [\delta_1,\delta_2]$
with $0<\delta_1<\delta_2<1$,
\end{theorem}
with the definition 
\begin{definition}
For all positive real $\tau$ we denote by 
$A(\mathbb T, \tau)$ the set 
$$A(\mathbb T, \tau)= \{h \in L^2(\mathbb T)\vert 
\sum_{s\in \mathbb Z} \vert s^\tau \hat{h}(s) \vert <\infty\}$$
\end{definition}
 This theorem has also been proved for $\alpha \in 
\mathbb N^*$ in \cite{RS04} and for $\alpha\in ]\frac{1}{2},+\infty[ 
\setminus \mathbb N^*$ in \cite{RS1111}.\\
 
\vspace{1cm}

 The second part of the present paper is devoted to the inversion of a class of Toeplitz matrices. We give an asymptotic expansion of 
$\left(T_{N} (f_{\alpha}) \right)_{k+1,l+1}^{-1}$ for $\alpha
\in ]0, \frac{1}{2}]$ and 
$\frac{k}{N} \rightarrow x$, $\frac{l}{N} \rightarrow y$ and $0<x\neq y<1$. First
we obtain these entries as a function of $ \cos(l-k)\theta_{0}$ and 
$\left(T_{N} (h_{\alpha}) \right)_{k+1,l+1}^{-1}$. It is Theorem \ref{TOEP1}. With the same hypothesis as for Theorem \ref{PREDIZERO} we have stated in  \cite{RS10} the following Theorem 
\begin{theorem} [\cite{RS10}] \label{TOEPMOINSDEUX}
For $0 <\alpha<\frac{1}{2}$ we have 
$$ c(1) \left( T_N ( h_{\alpha})\right)^{-1} _{[Nx]+1,[Ny]+1} =
N^{2\alpha-1} \frac{1}{\Gamma^2(\alpha)} G_\alpha (x,y) 
+ o(N^{2\alpha-1})$$
uniformly in $(x,y)$ for 
$0<\delta_1\le x\neq y<1.$
\end{theorem}
Theorem \ref{TOEPMOINSDEUX} has been proved for $\alpha \in 
\mathbb N^*$ in \cite{RS04}, for $\alpha=\frac{1}{2}$
in \cite{RS10} and for $\alpha\in ]\frac{1}{2},+\infty[ 
\setminus \mathbb N^*$ in \cite{RS1111}. 
The quantities
 $G_{\alpha}(x,y)$ is the integral kernel on $L^2(0,1)$ of Corollaries  \ref{TOEP2}. \\
  A direct consequence of theorems \ref{TOEP1} and \ref{TOEPMOINSDEUX} is that,
for $\alpha> 0$ the entries 
of $\left(T_{N} ( f_{\alpha}) \right)^{-1}$ are functions of $ \cos(l-k)\theta_{0}$ and the integral kernel $G_{\alpha}(x,y)$  (see corollaries \ref{TOEP2}).\\
  The  results of this paper are of interest in the analysis of time series. Indeed it is known that the $n$-th covariance matrix of a time series is a positive Toeplitz matrix.
  If $\phi$ is the symbol of this Toeplitz matrix, $\phi$ is called the spectral density of 
the time series. The time series with spectral density is the function $f_{\alpha}$ are also called  
GARMA processes. For more on this processes we refer the reader to 
\cite{ {B}, {RamBeau}, {WWCG}}
and to \cite{{Dahlhaus}, {DOT},  {GS}, {B}, {BrDa}, {KIRTEY}, {YILUHU}} for Toeplitz matrices in times series. \\
\textbf{Predictor polynomial}\\
Now we have to precise the deep link between 
the orthogonal polynomials and the inverse of the Toeplitz matrices.\\
Let $T_n(f)$ a Toeplitz matrix with symbol $f$ and $(\Phi_{n})_{n\in \mathbb N}$ the orthogonal polynomials 
with respect to $f$ (\cite{Ld}). To have the polynomial used for the prediction theory we put 
\begin{equation} \label{predizero1} 
\Phi^*_n(z)=\sum_{k=0}^n\frac {(T_n(f))^{-1}_{k+1,N+1}}
{(T_n(f))^{-1}_{N+1,N+1}}z^{k},~\mid z\mid =1.
\end{equation}
We define the polynomial $\Phi^*_n$ (see \cite{BS1}) as  
 \begin{equation} \label{predi} 
  \Phi_n ^* (z) = z^n \bar \Phi_n (\frac{1}{z}),
\end{equation}
that implies, with the symmetry of the Toeplitz matrix
\begin{equation} \label{predizero}
\Phi^*_n(z)=\sum_{k=0}^n\frac {(T_n(f))^{-1}_{k+1,1}}
{(T_n(f))^{-1}_{1,1}}z^{k},~\mid z\mid =1.
\end{equation}
The polynomials $P_{n}=\Phi_n ^* \sqrt{(T_n(f))^{-1}_{1,1}}$ are often called predictor polynomials. As we can see in the previous formula their coefficients are, up to a normalisation, the entries of the first
column of ${T_n (f)}^{-1}.$ \\
  
The proof of Theorem \ref{TOEP1} uses the important following theorem (\cite{Ld}),
\begin{theorem}\label{polpred}
If $h$ a non negative symbol with a finite set of zeroes, and $P_{n}$ the predictor polynomial 
of degree $n$ of $h$, we have, for all integers 
$s$ such that $-n \le s \le n$, 
$$ \widehat {\frac{1}{\vert P_{n}\vert^2}} (s)
= \hat h(s).$$
\end{theorem}
It implies
\begin{corollary}
For a fonction $h$ as in Theorem \ref{polpred} 
we have 
$$T_n \left( \frac{1}{\vert P_{n}\vert^2}\right) 
= T_n (h).$$
\end{corollary}

\section{Main results}
\subsection{Main notations}

In all the paper we consider the symbol defined by 
$\theta \mapsto 2^{2\alpha} (\cos \theta -\cos \theta_{0})^{2\alpha} c_{1}$ where 
$c_{1}= \frac{\vert P\vert}{\vert Q \vert}$ with $P,Q \in \mathbb R [X]$,  without zeros on the united circle 
and $-\frac{1}{2}<\alpha <\frac{1}{2}$ and $0<\theta_{0}<\pi.$ 
 We have $c_{1}=c_{1,1}\bar c_{1,1}$ with $c_{1,1} =\frac{P}{Q}$. Obviously 
 $c_{1,1}\in H^{2+}(\mathbb T)$ since $H^{2+} (\mathbb T)=\{ h \in L^2(\mathbb T)\vert u<0 \implies\hat h (u) =0 \}$. If $\chi$ is the function $\theta \mapsto e^{i\theta}$ and if
 $\chi_{0}= e^{i\theta_{0}} $
we put $g_{\alpha,\theta_{0},c_{1}}= (\chi-\chi_{0}) ^{\alpha} (\chi-\overline{\chi_{0}})^{\alpha} c_{1,1}$ and 
$g_{\alpha,\theta_{0}}= (\chi-\chi_{0}) ^{\alpha} (\chi-\overline{\chi_{0}})^{\alpha}$ since$\left(2(\cos \theta-\cos
\theta_{0})\right)^{2\alpha} = \vert \chi-\chi_{0}\vert^{2\alpha} \vert \chi-\overline{\chi_{0}}\vert ^{2\alpha}$.
Clearly $g_{\alpha,\theta_{0},c_{1}}, g_{\alpha,\theta_{0}} \in H^{2+} (\mathbb T)$ and 
$ 2^{2\alpha} (\cos \theta -\cos \theta_{0})^{2\alpha} c_{1} = g_{\alpha,\theta_{0},c_{1}} \overline{g_{\alpha,\theta_{0},c_{1}}}$, $ 2^{2\alpha} (\cos \theta -\cos \theta_{0})^{2\alpha}  = g_{\alpha,\theta_{0}} \overline{g_{\alpha,\theta_{0}}}$.
Then we denote by 
$ \beta_{k,\theta_{0},c_1}^{(\alpha)}$ the Fourier coefficient of $g^{-1}_{\alpha,\theta_{0},c_{1}}$ 
and by
$ \beta_{k,\theta_{0}}^{(\alpha)}$ the one of 
$g^{-1}_{\alpha,\theta_{0}}$. 
Without loss of generality 
we assume $\beta_{0,\theta_{0},c_{1}}^{(\alpha)}=1$.
We put also
${\tilde\beta_{k} }^{(\alpha)}= \widehat{\tilde g_{\alpha}^{-1}}(k)$ with 
$\tilde g_{\alpha}= (1-\chi)^{\alpha}$.
\subsection {Orthogonal polynomials}
\begin{theorem} \label{COEF}
Assume 
$\theta_0 \in ]0,\pi[$ and $-\frac{1}{2} <\alpha<\frac{1}{2}.$ Then we have for all integers $k$,
 $\frac{k}{N}\rightarrow x$, $0<x<1$, the 
asymptotic 
\begin{align*}
&\left( T_{N}^{-1}\left(\vert \chi-\chi_{0}Ê\vert ^{2\alpha}\vert \chi-\bar\chi_{0}Ê\vert ^{2\alpha} c_{1} \right) 
\right)_{k+1,1}
=\\ 
=&K_{\alpha,\theta_{0},c_1}  \cos
 \left(k\theta_{0} + \omega_{\alpha,\theta_0} \right) 
\left( T_{N}^{-1} \left( \vert \chi-1\vert ^{2\alpha} \right)\right)_{k+1,1} \left(1+o(1)\right)
\end{align*}
uniformly in $k$ for $x \in [\delta_{0},\delta _{1}]$, $ 0<\delta_{0}<\delta _{1}<1,$
and with $\omega_{\alpha,\theta_0}=\alpha \theta_0  + \arg \left( c_{1,1} (\theta_0)\right)- \frac{\pi\alpha}{2}$ and 
$K_{\alpha,\theta_{0},c_1} = 2^{-\alpha+1} (\sin \theta_{0})^{-\alpha} \sqrt{c_{1}^{-1} (\chi_0)}$ .
\end{theorem}
Then the  following statement is an  obvious consequence of Theorems \ref{COEF} and \ref{PREDIZERO}.
\begin{corollary}\label{DEUX}
With the same hypotheses as in Theorem \ref{COEF} we have 
\begin{align*}
&\left( T_{N}^{-1}\left(\vert \chi-\chi_{0}Ê\vert ^{2\alpha}\vert \chi-\bar\chi_{0}Ê\vert ^{2\alpha} c_{1}\right)
\right)_{k+1,1}=\\
&=\frac{K_{\alpha,\theta_{0},c_1}}{\Gamma(\alpha)} \cos \left(k\theta_{0}+\omega_{\alpha,\theta_0} \right) 
 k^{\alpha-1} (1-\frac{k}{N}) ^\alpha + o(N^{\alpha-1})
\end{align*}
uniformly in $k$ for $x \in [\delta_{0},\delta _{1}]$ $ 0<\delta_{0}<\delta _{1}<1$.
\end{corollary}
Moreover the equalities (\ref{predi}) and 
(\ref{predizero}) provide 
\begin{corollary}\label{GEGEN}
Let $\Phi_{N} = \displaystyle{ \sum_{j=0}^N \delta _j
\chi^j}$ be the orthogonal polynomial of degree $N$ (Gegenbauer  polynomial) with respect to the weight 
$\theta\mapsto 2^{2\alpha} (\cos \theta -\cos \theta_0) c_1(\theta)$, with $-\frac{1}{2}<\alpha<\frac{1}{2}$. Then we have, 
for $\frac{j}{N} \rightarrow x$, $0< x <1$,
$$
 \delta_j = N^{\alpha-1}\frac{K_{\alpha,\theta_{0},c_1}}{\Gamma(\alpha)}
\cos\left(N-j\theta_0 +\omega_{\alpha,\theta_0}\right) j^{\alpha} 
(1-\frac{j}{N})^{\alpha-1}  + o ( N^{\alpha-1}).
$$
uniformly in $j$ for $x \in [\delta_{0},\delta _{1}]$, $ 0<\delta_{0}<\delta _{1}<1.$
\end{corollary}
We can also point out the asymptotic of the coefficients of order $k$ of the predictor polynomial 
when $\frac{k}{N} \rightarrow 0$. 
\begin{theorem} \label{COEF2}
With the same hypotheses as in Theorem \ref{COEF} we have, if 
$\displaystyle{ \frac{k}{N} \rightarrow 0}$ when $N$ goes to the infinity 
$$ 
\left( T_{N}^{-1}\left(\vert \chi-\chi_{0}Ê\vert ^{2\alpha}\vert \chi-\bar\chi_{0}Ê\vert ^{2\alpha} c_{1}\right)
\right)_{k+1,1}
=  \beta_{k,\theta_{0},c_{1}}^{(\alpha)} +O(\frac{1}{N}).
$$
\end{theorem}
Lastly when $\alpha$ approaches $\frac{1}{2}$ we obtain the entries of the last column of 
$T_{N}\left(2 (\cos \theta - \cos \theta_{0}) c_{1}\right) $.

\begin{corollary}\label{DEMI}
Assume 
$\theta_0 \in ]0,\pi[$. Then for all integers $k$ for $\frac{k}{N}\rightarrow x$, $0<x<1$, we have the 
asymptotic 
\begin{align*}
&\left( T_{N}^{-1}\left(\vert \chi-\chi_{0}Ê\vert \vert \chi-\bar\chi_{0}Ê\vert  c_{1}\right) 
\right)_{k+1,1}
=\\ 
=&K_{1/2,\theta_{0},c_1}  \cos
 \left(k\theta_{0} + \omega_{1/2,\theta_0} \right) \sqrt {\frac{1}{k}-\frac{1}{N}} +o(\sqrt N)
 \end{align*}
 uniformly in $k$ for $x\in [\delta _{0}, \delta _{1}]$, $0<\delta _{0}<\delta _{1}<1$. 
 \end{corollary}
 \begin{remark}\label{DEUXMIMI}
 This corollary implies that the coefficient of order $k$ of the orthogonal polynomial with respect of 
 $\theta\mapsto 2 (\cos \theta - \cos \theta_{0}) c_{1}(\theta)$ is $ K_{1/2,\theta_{0},c_1}  \cos
 \left(k\theta_{0} + \omega_{1/2,\theta_0} \right) \left(\frac{1}{k}-\frac{1}{N}\right)^{-1} +o(\sqrt N)$
 \end{remark}
\subsection{Application to Toeplitz matrices}
\begin{theorem} \label{TOEP1}
Assume $\theta_{0}\in ]0,\pi[$ and $0<\alpha<\frac{1}{2}$. For $\frac{k}{N}\rightarrow x$,
 $\frac{l}{N}\rightarrow y$ and 
$0<x\neq y<1$, we have asymptotic
\begin{align*}
&\left( T_{N}^{-1}\left(\vert \chi-\chi_{0}Ê\vert ^{2\alpha}\vert \chi-\bar\chi_{0}Ê\vert ^{2\alpha} c_{1}\right) 
\right)_{k+1,l+1}
=\\ 
=& \vert K_{\alpha,\theta_0,c_1}\vert ^2 
 \cos \left(\theta_{0}(k-l) \right) 
\left( T_{N}^{-1} \left( \vert \chi-1\vert ^{2\alpha}  \right)\right)_{k+1,l+1} +o(N^{2\alpha-1})
\end{align*}
uniformly for $k,l$ such that $0<\delta _{1}<x\neq y <\delta _{2}<1.$
\end{theorem}
At it has been said in the introduction this statement and the results of \cite{RS10} provides the next corollary.
\begin{corollary}\label{TOEP2}
Assume $\alpha\in ]0,\frac{1}{2}]$ and $\theta_{0} \in ]0,\pi[$. Let $G_{\alpha}$ be the function defined on $0<x \neq y< 1$ by 
$$ G_{\alpha}(x,y) =\frac{ x^\alpha y^\alpha}{\Gamma^2(\alpha)}
\int_{\max(x,y)}^1 \frac{(t-x)^{\alpha-1} (t-y)^{\alpha-1} }{t^{2\alpha}} dt .$$
With the same hypothesis as in Theorem \ref{TOEP1} we have the asymptotic
\begin{align*}
&\left( T_{N}^{-1}\left(\vert \chi-\chi_{0}Ê\vert ^{2\alpha}\vert \chi-\overline{\chi_{0}}Ê\vert ^{2\alpha} c_{1}\right) 
\right)_{[Nx]+1,[Ny]+1}
=\\ 
=& N^{2\alpha-1} \vert K_{\alpha,\theta_{0},c_1}\vert ^2 \cos \left(\theta_{0}([Nx]-[Ny]) \right) 
G_{\alpha}(x,y)  +o(N^{2\alpha-1})
\end{align*}
uniformly in $k,l$ for $0<\delta _{1} \le x\neq y\le \delta _{2}<1$.
\end{corollary}

\subsection{Jacobi polynomial (in a particular case)}
We note that in Theorem \ref{TOEP1} one passes from the zeroes $\chi_0$ and $\overline{\chi_0}$ to two 
zeroes $\chi_1=e^{i \theta_1}$ and 
$\chi_2=e^{i \theta_2}$ with 
$\vert\theta_1-\theta_2\vert \in ]0,\pi[$. Namely it is easy to see that  
\begin{align*}
 &T_{N}^{-1}(\vert \chi-\chi_{1}\vert ^{2\alpha}\vert \chi-\chi_{2}\vert ^{2\alpha}c_{1})
=\\
& \Delta(\chi_1^{1/2} \chi_2^{1/2})  T_N^{-1}\left( \vert  \chi_{1}^{1/2} \chi_{2}^{-1/2} -\psi \vert ^{2\alpha}
\vert \chi_{1}^{-1/2} \chi_{2}^{1/2} -\psi \vert^{2\alpha} c_{1,\psi}\right)^{-1} 
\Delta^{-1} (\chi_1^{1/2} \chi_2^{1/2})
\end{align*}
with 
$\Delta(\chi_1^{1/2} \chi^{1/2}) $ is the diagonal matrix defined by 
$\left(\Delta(\chi_1^{1/2}\chi_2^{1/2})\right)_{i,j} =0$ if $i\neq j$ and 
$\left(\Delta( \psi)\right)_{j,j}=(\chi_1^{1/2} \chi_2^{1/2})^j$.

From this and Equation (\ref{predi}) we deduce 
the following proposition
\begin{proposition}
Let $\Phi_{1,2} = \displaystyle{ \sum_{j=0}\tilde  \delta_j
\chi^j}$ be the orthogonal polynomial (Jacobi polynomial) with respect to the weight 
$ \vert \chi - \chi_1 \vert^{2\alpha} \vert \chi -\chi_2
\vert^{2\alpha}$, with 
$\alpha\in]-\frac{1}{2}, \frac{1}{2}]$. Let $K_{\alpha, \theta_1,\theta_2}$ be the real  
$2^{-\alpha+1} \vert \sin (\theta_1-\theta_2)\vert^{-\alpha}.$ Then we have, 
for $\frac{j}{N} \rightarrow x$, $0< x <1$.
$$
\tilde \delta_j = N^{\alpha-1}K_{\alpha, \theta_1,\theta_2}(\overline{(\chi_1 \chi_2)^{1/2}})^{N-j} 
\cos\left((\frac{\theta_1-\theta_2}{2}) (N-j)+
\omega_{\alpha,\theta_1-\theta_{2}} \right)
 \frac{j^\alpha (1-\frac{j}{N})^{\alpha-1}}{ {\Gamma (\alpha)}}+ o ( N^{\alpha-1}),
$$
uniformly in $j$ for $x\in [\delta _{0}, \delta _{1}]$, $0<\delta _{0}<\delta _{1}<1$.
\end{proposition}
\section{Inversion formula}
\subsection{Definitions and notations}
Let $H^{2+}(\mathbb T)$ and 
$H^{2-}(\mathbb T)$ the two subspaces of $L^2(\mathbb T)$ defined by 
$H^{2+} (\mathbb T)=\{ h \in L^2(\mathbb T)\vert u<0 \implies \hat h (u) =0 \}$ and
$H^{2-} (\mathbb T)=\{ h \in L^2(\mathbb T)\vert u\ge0 \implies \hat h (u) =0 \}$.
We denote by $\pi_{+}$ the orthogonal projector on $H^{2+}(\mathbb T)$ 
and $\pi_{-} $ the orthogonal projector on $H^{2-}(\mathbb T)$.
 It is known (see \cite{GS}) that if $f\ge 0$ and $\ln f \in L^1(\mathbb T)$ we have 
$ f = g \bar g$ with $g\in H^{2+} (\mathbb T)$. Put $\Phi_{N}= \frac{g}{\bar g} \chi^{N+1}$. Let $H_{\Phi_{N}}$ and 
$H^*_{\Phi_{N}}$ be the two Hankel operators defined respectively on $H^{2+}$ and 
$H^{2-}$ by 
$$ H_{\Phi_{N}} \, : \quad H^{2+}(\mathbb T)\rightarrow H^{2-}(\mathbb T), 
\quad \quad H_{\Phi_{N}} (\psi )= \pi_{-}( \Phi_{N}\psi),$$
 and 
 $$H^*_{\Phi_{N}}\, :  \quad  H^{2-}(\mathbb T)\rightarrow 
H^{2+}(\mathbb T), \quad \quad H^*_{\Phi_{N}} (\psi )= \pi_{+}( \bar\Phi_{N}\psi).$$
\subsection{A generalised inversion formula}
 We have stated in \cite{RS10} for a precise class of non regular functions which contains  $\cos^\alpha (\theta-\theta_{0}) c_{1}$ and $\left( \cos \theta-\cos\theta_{0}\right)^\alpha c_{1}$ the following lemma (see the appendix of \cite{RS10} for the demonstration),  
\begin{lemma}\label{INVERS}
Let $f$ be an almost everywhere positive function on the torus $\mathbb T$ such that 
$ \ln f$, $f$, and $\frac{1}{f}$ are in $\mathbb L^1(\mathbb T)$. Then $f=g \bar g$ with 
$g\in H^{2+}(\mathbb T)$. For all trigonometric polynomials $P$ of degree at most $N$, 
we define $G_{N,f}(P)$ by 
$$ G_{N,f}(P) = \frac{1}{g} \pi_{+} \left( \frac{P}{\bar g} \right)-
\frac{1}{g} \pi_{+} \left( \Phi_{N}\sum_{s=0}^\infty \left( H^*_{\Phi_{N}} H_{\Phi_{N}} \right)^s 
\pi_{+} \bar \Phi_{N}\pi_{+}\left( \frac{P}{\bar g}\right) \right).$$
For all $P$ we have 
\begin{itemize}
\item The serie $\displaystyle{\sum_{s=0}^\infty \left( H^*_{\Phi_{N}} H_{\Phi_{N}} \right)^s 
\pi_{+} \bar \Phi_{N}\pi_{+}\left( \frac{P}{\bar g}\right)}$ converges in $L^2(\mathbb T)$.
\item
$\det \left(T_{N}(f)\right) \neq 0$ and 
$$ \left( T_{N}(f)\right)^{-1} (P) =G_{N,f}(P).$$
\end{itemize}
\end{lemma}
An obvious corollary of Lemma \ref{INVERS} is 
\begin{corollary} \label{INVERS2}
With the hypotheses of Lemma \ref{INVERS} we have 
$$ \left( T_{N}(f)\right)^{-1} _{l+1,k+1}=
\Bigl \langle \pi_{+}\left( \frac{\chi^k}{\bar g}\right)\Big \vert \left( \frac{\chi^l}{\bar g}\right)\Bigr \rangle
- \Bigl \langle \sum_{s=0}^\infty \left( H^*_{\Phi_{N}} H_{\Phi_{N}} \right)^s 
\pi_{+} \bar \Phi_{N}\pi_{+}\left( \frac{\chi^k}{\bar g}\right) \Big \vert \bar \Phi_{N} 
\left( \frac{\chi^l}{\bar g}\right) \Bigr \rangle.
$$
\end{corollary}
Lastly if $\gamma_{u,\alpha,\theta}= \widehat { \frac{g_{\alpha,\theta_{0}}}{\overline{ g_{\alpha,\theta_{0}}}}}(u)$ we obtain as in \cite{RS10} the formal result 
\begin{align*}
\left( H_{\Phi_{N}}^* H_{\Phi_{N}}\right)^m \pi_{+}\bar  \Phi_{N} \pi_{+}\left( \frac{\chi^k}{\bar g}\right)
&=
\sum_{u=0}^k \beta_{u,\theta_{0},c_{1}} ^{(\alpha)} \sum _{n_{0}=0}^\infty \left(  \sum _{n_{1}=1}^\infty
\bar \gamma_{-(N+1+n_{1}+n_{0}),\alpha,\theta_{0}}\right.\\
& \sum _{n_{2}=0}^\infty \gamma_{-(N+1+n_{1}+n_{2}),\alpha,\theta_{0}} \cdots 
  \sum _{n_{2m-1}=1}^\infty \bar \gamma_{-(N+1+n_{2m-1}+n_{2m-2}),\alpha,\theta_{0}}\\
 & \left.  \sum _{n_{2m}=0}^\infty\gamma_{-(N+1+n_{2m-1}+n_{2m}),\alpha,\theta_{0}} 
  \bar \gamma_{-(u-(N+1+n_{2m}),\alpha,\theta_{0})}\right) \chi^{n_{0}}
\end{align*}
\subsection{Application to the orthogonal polynomials}
With the corollary \ref{INVERS2} and the hypothesis on $\beta^{(\alpha)} _{0,\theta_{0},c_{1}}$ 
the equality in the corollary \ref{INVERS2} becomes, for $l=1$, and for $f=\vert \chi-\bar \chi_0\vert ^{2\alpha} \vert \chi - \bar \chi_0\vert^{2\alpha} c_1$
\begin{equation} \label{STAR}
\left( T_{N}(f)\right)^{-1} _{1,k+1}= \beta^{(\alpha)} _{k,\theta_{0},c_{1}} - \sum_{u=0}^k \beta_{k-u,\theta_{0},c_{1}} ^{(\alpha)}  H_{N}(u)
\end{equation}
with
\begin{align*}
H_{N}(u)
&=
\sum_{m=0}^{+\infty}\left( \sum _{n_{0}=0}^\infty\gamma_{N+1+n_{0},\alpha,\theta_{0}} \left(  \sum _{n_{1}=0}^\infty
\bar \gamma_{-(N+1+n_{1}+n_{0}),\alpha,\theta_{0}}\right.\right.\\
& \sum _{n_{2}=0}^\infty \gamma_{-(N+1+n_{1}+n_{2}),\alpha,\theta_{0}} \cdots 
  \sum _{n_{2m-1}=0}^\infty \bar \gamma_{-(N+1+n_{2m-1}+n_{2m-2}),\alpha,\theta_{0}}\\
 & \left. \left. \sum _{n_{2m}=0}^\infty\gamma_{-(N+1+n_{2m-1}+n_{2m}),\alpha,\theta_{0}} 
  \bar \gamma_{(u-(N+1+n_{2m}),\alpha,\theta_{0}}\right) \right)
\end{align*}
Our proof consists in the computation of the coefficients $\beta_{u,\theta_{0},c_{1}} ^{(\alpha)}$,
$\gamma_{u,\alpha,\theta}$ and $H_{N}(u)$ which appear in the inversion formula. For each step 
we obtain the corresponding terms for the symbol $ 2^\alpha (1-\cos \theta) c_1$ multiplied by 
a trigonometric coefficient. That provides the expected link with the formulas in Theorems 
\ref{PREDIZERO}, \ref{TOEPMOINSDEUX}.
\section{Demonstration of Theorem \ref{COEF}}
\subsection{Asymptotic of $\beta_{k,\theta_0,c_{1}}^{(\alpha)}$}
\begin{remark}
In the rest of this paper we denote by $c_{1,1}$
the function in $H^{2+}(\mathbb T)$ such that 
$c_1 =c_{1,1} \overline{c_{1,1}}$. In all this proof we put 
$\phi_{0} =
\arg \left(c_{1,1} (\theta_0)\right)$ 
\end{remark}
\begin{prop} \label{PROP1}
For $-\frac{1}{2}<\alpha< \frac{1}{2}$ and $\theta_0
\in]0,\pi[$ we have, for sufficiently large 
$k$ and for the real $\beta$ defined by 
$\beta=\alpha-\frac{1}{2}$ if $\alpha<0$ and $\beta=\alpha$ if $\alpha>0$,
$$ 
\beta_{k,\theta_0,c_1} ^{(\alpha)}
=K_{\alpha,\theta_{0},c_{1}}
\cos (k\theta_{0}+\omega_{\alpha,\theta_0}) 
\frac{k^{\alpha-1}}{\Gamma (\alpha)}+o(k^{\beta-1})
$$
uniformly in $k$.
With 
$K_{\alpha,\theta_{0},c_{1}} =  \frac{1}{\sqrt{c_{1}(\chi_0)}} 2^{-\alpha+1} (\sin \theta_0 )^{-\alpha}$
and
$ \omega_{\alpha,\theta_0}$ as in the statement of 
Theorem \ref{COEF}.
\end{prop}
First we have to prove the lemma 
\begin{lemma} \label{PRELI}
 For $-\frac{1}{2}<\alpha< \frac{1}{2}$ and $\theta_0
\in ]0 ,\pi[$ we have, for a sufficiently large $k$. 
$$ 
\beta_{k,\theta_0}^{(\alpha)}  
= K_{\alpha,\theta_{0}} 
\cos ((k+\alpha) \theta_0 +\omega_\alpha) 
\frac{k^{\alpha-1}}{\Gamma (\alpha)}+o(k^{\beta-1}).
$$
uniformly in $k$,
with $ K_{\alpha,\theta_{0}} = 2^{-\alpha+1} (\sin \theta_0)^{-\alpha}$,
$\omega_\alpha =-\frac{\pi \alpha}{2}$, and $\beta$
as in Property \ref{PROP1}.
\end{lemma}
\begin{remark}
In these two last statements``uniformly in $k$ '' means that for all $\epsilon>0$ we have an integer 
$k_{\epsilon}$ such that for all $k\ge k_{\epsilon}$
$$\Bigl \vert  \beta_{k,\theta_0}^{(\alpha)}  
- K_{\alpha,\theta_{0}} 
\cos ((k+\alpha) \theta_0 +\omega_\alpha) 
\frac{k^{\alpha-1}}{\Gamma (\alpha)}\Bigr \vert <\epsilon k^{\beta-1}$$
and 
$$\Bigl \vert  \beta_{k,\theta_0,c_{1}}^{(\alpha)}  
- K_{\alpha,\theta_{0},c_{1}} 
\cos ((k+\alpha) \theta_0 +\omega_\alpha) 
\frac{k^{\alpha-1}}{\Gamma (\alpha)}\Bigr \vert <\epsilon k^{\beta-1}.$$
\end{remark}
\begin{proof}{} 
 With our notations we 
can write 
$$ 
\beta_{k,\theta_{0}}^{(\alpha)} =
\sum_{u=0}^k {\tilde \beta_u}^{(\alpha)} (\chi_0)^u 
{\tilde\beta_{k-u}} ^{(\alpha)} ( \overline{\chi_0})^{k-u}.
$$
Put $k_0= k^\gamma$  with $0<\gamma<1$ such that for $u>k_{0}$ we have  
\begin{equation} \label{ZYG}
 {\tilde \beta_u}^{(\alpha)} =\frac{ u^{\alpha-1} }{\Gamma (\alpha)}+O(k^{\alpha-2})
\end{equation}
uniformly in $u$ (see \cite{Zyg2}).
Writting for $k\ge k_{0}$ 
\begin{align*}
\sum_{u=0}^k {\tilde \beta_u}^{(\alpha)} (\chi_0)^u 
{\tilde\beta_{k-u}} ^{(\alpha)} ( \overline{\chi_0})^{k-u}
&= \sum_{u=0}^{k_0} {\tilde \beta_u}^{(\alpha)} (\chi_0)^u 
{\tilde\beta_{k-u}} ^{(\alpha)}( \overline{\chi_0})^{k-u}\\
&+ \sum_{u=k_0+1}^{k-k_0-1} {\tilde \beta_u}^{(\alpha)} (\chi_0)^u 
{\tilde\beta_{k-u}} ^{(\alpha)} ( \overline{\chi_0})^{k-u}\\
&+ \sum_{u=k-k_0}^k {\tilde \beta_u}^{(\alpha)} (\chi_0)^u 
{\tilde\beta_{k-u}} ^{(\alpha)} ( \overline{\chi_0})^{k-u}.
\end{align*}
The first sum is also 
$$ \sum_{u=0}^{k_0} {\tilde \beta_u}^{(\alpha)} (\chi_0)^u 
\left( {\tilde\beta_{k-u}} ^{(\alpha)}  - {\tilde\beta_{k}} ^{(\alpha)}+{\tilde\beta_{k}} ^{(\alpha)}\right)( \overline{\chi_0})^{k-u}.$$
We observe that 
\begin{equation} \label{MAJOR1}
\Bigl \vert \sum_{u=0}^{k_0} {\tilde \beta_u}^{(\alpha)} (\chi_0)^u 
\left( {\tilde\beta_{k-u}} ^{(\alpha)}  - {\tilde\beta_{k}} ^{(\alpha)}\right)\Bigr \vert 
\le \frac{1}{\Gamma(\alpha)}  \sum_{u=0}^{k_{0}} \vert (k-u)^{\alpha-1} - k^{\alpha-1}\vert
\vert \tilde \beta_{u}^{(\alpha)} \vert.
\end{equation}
Consequently
\begin{align*} 
\sum_{u=0}^{k_0} {\tilde \beta_u}^{(\alpha)} (\chi_0)^u 
{\tilde\beta_{k-u}} ^{(\alpha)} ( \overline{\chi_0})^{k-u}
&=
\left( \sum_{u=0}^{k_0} \beta_u^{(\alpha)} 
(\chi_0)^{2u}\right) {\bar\chi_{0}}^k
\frac{k^{\alpha-1}}{\Gamma (\alpha)} +R_{1}
\\
&= \left( \sum_{u=0}^{+\infty} {\tilde\beta_u}^{(\alpha)} 
(\chi_0)^{2u} - \sum_{u=k_0}^{+\infty}{\tilde \beta_u}^{(\alpha)} 
(\chi_0)^{2u}\right) {\bar\chi_{0}}^k
\frac{k^{\alpha-1}}{\Gamma (\alpha)} 
+R_{1}
\end{align*}
with $R_{1}= O(k^{\alpha-2+\gamma})$ if $\alpha<0$ and $R_{1}= O(k^{\alpha-2+\gamma\alpha})$ 
if $\alpha>0$.
Then 
Lemma \ref{APPENDIX1} implies 
\begin{equation}\label{MAJOR2}
 \vert \sum_{u=k_0}^{+\infty} {\tilde\beta_u}^{(\alpha)} 
(\chi_0)^{2u}\vert  \le \vert {\tilde\beta_{k_{0}}}^{(\alpha)} \chi_{0}^{2u} \vert +
\sum_{u=k_{0}} ^{\infty} \frac{\vert {\tilde \beta_{u+1}}^{(\alpha)} -{\tilde \beta}_{u}^{(\alpha)} \vert}{\Gamma(\alpha)},
\end{equation}
that is 
$$ \sum_{u=k_0}^{+\infty} {\tilde\beta_u}^{(\alpha)} 
(\chi_0)^{2u} =O(k_0^{\alpha-1}).$$
Finally we get
$$ \sum_{u=0}^{k_0} {\tilde\beta_u}^{(\alpha)} (\chi_0)^u 
{\tilde\beta_{k-u}} ^{(\alpha)} ( \overline{\chi_0})^{k-u}
=\frac{k^{\alpha-1}}{ \Gamma (\alpha)} 
\overline{ \chi_0}^k  (1-\chi_0^2)^{-\alpha}+
O\left( k^{(\alpha-1)(\gamma+1)}\right) +R_{1}.
$$
Analogously we obtain
$$\sum_{u=k-k_0}^k {\tilde\beta_u}^{(\alpha)} (\chi_0)^u 
{\tilde\beta_{k-u} }^{(\alpha)}  ( \overline{\chi_0})^{k-u} =
\chi_0^k \frac{k^{\alpha-1}}{\Gamma (\alpha)}
(1-\bar \chi_0^2)^{-\alpha} +
O\left( k^{(\alpha-1)(\gamma+1)}\right) +R_{2},$$
with $R_{2}$ as $R_{1}$. \\
For the third sum an Abel summation provides 
\begin{align*}
\sum_{u=k_0+1}^{k-k_0-1} {\tilde\beta_u}^{(\alpha)}  (\chi_0)^u 
{\tilde\beta_{k-u}} ^{(\alpha)}  ( \overline{\chi_0})^{k-u}&=
\overline{ \chi_0} ^k \left( {\tilde\beta_{k_0}}^{(\alpha)}  {\tilde\beta_{k-k_0} }^{(\alpha)}
\sigma _{k_0-1} \right.\\
&+\left. \sum_{u=k_0} ^{k-k_0-2} ({\tilde\beta_u}^\alpha {\tilde\beta_{k-u} }^{(\alpha)} -
{\tilde \beta_{u+1}}^{(\alpha)} {\tilde\beta_{k-u-1} }^{(\alpha)})
 \sigma_u\right) +{\tilde \beta_{k-k_0-1}} ^{(\alpha)} 
{\tilde \beta_{k_0}}^{(\alpha)} \sigma _{k-k_0}
 \end{align*}
 with 
 $\sigma_v= 1+ \chi_0^2+ \cdot +\chi_0^{2v}$. 
This last sum is also equal to 
$ \overline{ \chi_0} ^k \left( A+B\right),$
 with
 $$ \vert A\vert = O \left({\tilde \beta}_{k_{0}}^{(\alpha)}{\tilde \beta_{k}}^{(\alpha)}\right) = 
 O(k_{0}^{\alpha-1} k^{\alpha-1})
 =o(k^{(\alpha-1)(\gamma+1)})$$
 and 
$$
 B = -\sum_{u=k_0} ^{k-k_0-2} \frac{1}{\Gamma^2(\alpha)} \left (u^{\alpha-1} (k-u)^{\alpha-1}
 - (u+1)^{\alpha-1} (k-u-1)^{\alpha-1} \right) 
 \frac{\chi_0^{2u+2}}{1-\chi_0^2}.$$ 
 The main value Theorem implies 
 \begin{equation} \label{MAJOR3}
 \vert B \vert \le M k \left( \sum_{v=k_0}^{k-k_0}
  v^{\alpha-2} (k-v)^{\alpha-2}\right),
  \end{equation}
  with $M$ no depending from $k$.
With the Euler and Mac-Laurin formula it is easyly seen that 
  $$ \sum_{v=k_0}^{k-k_0}
  v^{\alpha-2} (k-v)^{\alpha-1} \sim
  k_0^{\alpha-2} (k-k_0)^{\alpha-1} +  
  k_0^{\alpha-1} (k-k_0)^{\alpha-2} + 
  \int_{k_0}^{k-k_0} t^{\alpha-2} (k-t)^{\alpha-2} dt.$$
 The decomposition  
  $$  \int_{k_{0}}^{k-k_0} t^{\alpha-2} (k-t)^{\alpha-2} dt = 
  \int_{k_{0}}^{k/2} t^{\alpha-2} (k-t)^{\alpha-2} dt +\int_{k/2}^{k-k_0} t^{\alpha-2} (k-t)^{\alpha-2} dt 
 $$ provides the estimation 
 $\vert B \vert = O(k_{0}^{\alpha-1} k^{\alpha-1}) = O( k^{(\alpha-1)(\gamma+1)})$.
 If $\alpha>0$ and $0<\gamma<1$ we have 
 $$ \beta_{k,\theta_0}^{(\alpha)} =
\frac{k^{\alpha-1} } {\Gamma (\alpha)} 
\left( \overline{\chi_0^k} 
(1-\chi_0^2)^{-\alpha}+ \chi_0 ^k (1-\overline{ \chi_0^2})^{-\alpha} \right) +o(k^{\alpha-1})$$
If $\alpha<0$ and $\gamma=\frac{1}{2}$  we get 
 $$ \beta_{k,\theta_0}^{(\alpha)} =
\frac{k^{\alpha-1} } {\Gamma (\alpha)} 
\left( \overline{\chi_0^k} 
(1-\chi_0^2)^{-\alpha}+ \chi_0 ^k (1-\overline{ \chi_0^2})^{-\alpha} \right) +o(k^{\beta-1})$$
with $\beta=\alpha-\frac{1}{2}$. 
On the another hand we have 
\begin{align*}
\beta_{k,\theta_0}^{(\alpha)} &= 2 \frac{k^{\alpha-1} } {\Gamma (\alpha)} 
\Re \left( e^{-i k \theta_0} \left(1-\cos (2 \theta_0)-
i \sin (2\theta_0)\right) ^{-\alpha}\right)+o(k^{\beta-1})\\
&= 2^{1-\alpha} \frac{k^{\alpha-1} } {\Gamma (\alpha)} \Re\left( e^{-i k \theta_0}
\left( \sin (\theta_0) \left( \sin \theta_0 -i\cos \theta_0\right) \right)^{-\alpha}\right)+o(k^{\beta-1})
\end{align*}
Since $\theta_0\in ]0, \pi[$ we have 
$ \left( \sin \theta_0
( \sin \theta_0 -i \cos \theta_0) \right)^{-\alpha}
= (\sin \theta_0 )^{-\alpha} e^{i\alpha(\frac{\pi}{2}-\theta_0)} $\\
This last remark gives the definition of $\omega_\alpha$. The equations (\ref{MAJOR1}), 
(\ref{MAJOR2}), (\ref{MAJOR3}), imply the uniformity that completes the proof of the lemma.\end{proof}
To ends the proof of the property we need to obtain 
$\beta_{k,\theta_0,c_1} ^{(\alpha)}$ from 
$\beta_{k,\alpha,\theta_0} ^{(\alpha)}$ for a sufficiently large $k$. We can remark that a similar case has been  treated in \cite{RS1111} for the function $(1-\chi)^\alpha c_1$. Here we develop the same idea than in this last paper.
Let $c_{m}$ the coefficient of Fourier of order $m$ of the function $c_{1,1}^{-1}$. The hypotheses 
on $c_{1,1}$ imply that $c_{1,1}^{-1}$ is in $A(\mathbb T , p) =\{ h \in L^2 (\mathbb T) \vert
 \sum_{u\in \mathbb Z} u^p \vert \hat h (u)\vert <\infty \}$ for all positive integer $p$. 
 We have, 
 $\displaystyle{ \beta^{(\alpha)}_{m,\theta_{0},c_{1}} = \sum_{s=0}^m \beta^{(\alpha)}_{m,\theta_{0}} c_{m-s}}.$
 For $0<\nu<1$ we can write 
 $$ \sum_{s=0}^m \beta^{(\alpha)}_{s,\theta_{0}} c_{m-s} = \sum_{s=0}^{m-m^\nu} \beta^{(\alpha)}_{s,\theta_{0}} c_{m-s} + \sum_{s=m-m^\nu+1 }^m \beta^{(\alpha)}_{s,\theta_{0}} c_{m-s}.$$ 
Lemma \ref{PRELI} provides 
\begin{align*}
\sum_{s=m-m^\nu+1 }^m \beta^{(\alpha)}_{s,\theta_{0}} c_{m-s} &= \left(K_{\alpha,\theta_{0}}  \sum_{s=m-m^\nu} ^m \frac{s^{\alpha-1}} {\Gamma (\alpha)} (\cos \left( (s+\alpha) \theta_{0} +\omega_{\alpha}\right) c_{m-s} \right)\\
&+ o(m^{\beta-1}) \sum_{s=m-m^\nu+1 }^m \vert c_{m-s}\vert 
\end{align*}
and, since $\sum_{s\in \mathbb Z} \vert c_{s}\vert <\infty$, we have 
$$ \sum_{s=m-m^\nu+1 }^m \beta^{(\alpha)}_{s,\theta_{0}} c_{m-s} = K_{\alpha,\theta_{0}} \frac{s^{\alpha-1}} {\Gamma (\alpha)} \sum_{s=m-m^\nu} ^m\frac{s^{\alpha-1}} {\Gamma (\alpha)}  (\cos \left( (s+\alpha) \theta_{0} +\omega_{\alpha}\right) c_{m-s} + o(m^{\beta-1}).$$
We have always
\begin{equation} \label{MAJOR4}
\Bigl \vert \sum_{s=m-m^\nu}^m (s^{\alpha-1} - m^{\alpha-1})  c_{m-s}\Bigr \vert \le (1-\alpha) 
m^{\nu+\alpha-2} \sum_{s=m-m^\nu}^m \vert c_{m-s}\vert.
\end{equation}
The convergence of $(c_{s})$ implies 
\begin{align*}
& K_{\alpha,\theta_{0}}  \sum_{s=m-m^\nu} ^m \frac{s^{\alpha-1}-m^{\alpha-1}+m^{\alpha-1}} {\Gamma (\alpha)} (\cos \left( (s+\alpha) \theta_{0} +\omega_{\alpha}\right) c_{m-s} 
\\ &= K_{\alpha,\theta_{0}}  \frac{m^{\alpha-1}} {\Gamma (\alpha)}\sum_{s=m-m^\nu} ^m \ (\cos \left( (s+\alpha) \theta_{0} +\omega_{\alpha}\right) c_{m-s} + O( m^{\alpha-2+\nu}). 
\end{align*}
For all positive integer $p$ the function $c_{1,1}$ $ A(p,\mathbb T)$). Hence one can prove first  
\begin{equation}\label{MAJOR5}
 \Bigr \vert \sum_{v= m^\nu+1}^\infty e^{-iv\theta} c_{v} \Bigl \vert \le (m^{-p\nu}) \sum_{s\in \mathbb Z} 
 \vert c_{s}\vert 
\end{equation}
and secondly
\begin{align*} 
\sum_{s=m-m^\nu} ^m (\cos \left( (s+\alpha) \theta_{0} +\omega_{\alpha}\right) c_{m-s}&=
\frac{1}{2}\left( \sum_{s=m-m^\nu} ^m e^{i s \theta_{0}} c_{m-s}\right) 
e^{i(\theta_{0}\alpha+\omega_{\alpha})} \\
&+ \frac{1}{2}\left( \sum_{s=m-m^\nu} ^m e^{-i s \theta_{0}} c_{m-s}\right) 
e^{-i(\theta_{0}\alpha+\omega_{\alpha})}\\
&= \frac{1}{2} \left( c^{-1}_{1,1}(e^{-i\theta_{0}}) e^{i(m\theta_{0}+\theta_{0}\alpha+\omega_{\alpha})}
+c^{-1}_{1,1}(e^{i\theta_{0}}) e^{-i(m\theta_{0}+\theta_{0}\alpha+\omega_{\alpha})}\right)\\
&+O(m^{-p\nu}).
\end{align*}
Since $ \overline{c^{-1}_{1,1} (e^{i\theta_{0}})} =c^{-1}_{1,1} (e^{-i\theta_{0}})$  that 
last formula provides  
\begin{equation}\label{COS}
\sum_{s=m-m^\nu} ^m (\cos \left( (s+\alpha) \theta_{0} +\omega_{\alpha}\right) c_{m-s} 
 = \sqrt{ c^{-1}_{1} (\chi_{0}) } \cos \left((m+\alpha)\theta_{0}
 +\omega_{\alpha}+\phi_{0}\right) 
 +O(m^{-p\nu}) 
 \end{equation}
 and 
 \begin{align*}
  \sum_{s=m-m^\nu+1 }^m \beta^{(\alpha)}_{s,\theta_{0}} c_{m-s} &= 
  K_{\alpha,\theta_{0}} \frac{m^{\alpha-1} }{\Gamma (\alpha)} 
  \sqrt{ c^{-1}_{1} (\chi_{0}) } \cos \left((m+\alpha)\theta_{0}+\omega_{\alpha}+\phi_{0}\right)  
 \\& + O(m^{\alpha-1-p\nu})+ O( m^{\alpha-2+\nu})+ o(m^{\beta-1}).
   \end{align*}
 
 On the other hand we have (because $c_{1,1}^{-1}$ in $A(\mathbb T,p)$) 
 $$
 \Bigr \vert \sum_{s=0}^{m-m^\nu} \beta_{s,\alpha} 
 c_{m-s} \Bigl \vert 
\le 
 \frac{1}{m^{2\nu }} \sum_{v\in \mathbb Z} v^p \vert c_{v}\vert  
 \max_{s\in \mathbb N} (\vert \beta_{s,\theta_{0}}^{(\alpha)}\vert).
 $$
 For a good choice of $p$ and $\nu$ we obtain 
  the expected formula for $\beta_{\alpha,\theta_{0},c_{1}}.$ The uniformity is provided by Lemma \ref 
  {PRELI} and the equation (\ref{MAJOR4}) and (\ref{MAJOR5}).

\subsection{Estimation of the Fourier coefficients
of 
$\frac{g_{\alpha,\theta_0} }
{\overline{g_{\alpha,\theta_0}}}$}

\begin{prop} \label{prop2}
Assume $- \frac{1}{2} <\alpha<\frac{1}{2}$ and 
$ \theta_0 \in ]0, \pi[ $ then we have for all integer $k \ge 0$  sufficiently large
$$ 
\widehat{\frac{g_{\alpha,\theta_0} }
{\overline{g_{\alpha,\theta_0} }}} (-k)
= \frac{2}{k+\alpha} \frac { \sin (\pi\alpha)} {\pi} 
\cos ( \theta_0 k + 2\omega'_{\alpha,\theta_0}) 
 +o(k^{\min (\alpha-1, -1)})$$
 uniformly in $k$ and with $\omega'_{\alpha,\theta_{0}} =\phi_{\alpha}+\phi'_{0}$
 where $\phi'_{0}= \arg \left( \frac{c_{1,1}}{\bar c_{1,1}} \right) (e^{i\theta_{0}}) $ and 
 $\phi_{\alpha} = \arg \left( \frac{\chi_{0}^2-1}{{\bar \chi_{0}}^2 -1}\right)^\alpha$.
\end{prop} 
First we have to prove the lemma 
\begin{lemma} \label{PRELI2}
For $- \frac{1}{2} <\alpha<\frac{1}{2}$ and 
$ \theta_0 \in ]0,\pi[$ we have, for all integer $k$ sufficiently large 
$$ 
\gamma_{-k}=  \frac{2}{k+\alpha} \frac { \sin (\alpha)} {\pi} 
\cos \left( \theta_0 k + \phi_{\alpha}\right)) +o(k^{\min (\alpha-1, -1)}),
$$
uniformly in $k$ and where
$\gamma_{k}$ is the coefficient of order $k$ of the function $\frac{ (\chi \chi_{0}-1)^\alpha (\chi \bar \chi_{0}-1)^\alpha }{(\bar \chi \bar \chi_{0}-1)^\alpha(\bar \chi \chi_{0}-1)^\alpha}$.
 \end{lemma} 
\begin {proof}{of Lemma \ref{PRELI2}} 
In all this proof we denote respectively by $\tilde \gamma_{k}, \gamma_{1,k}, \gamma_{2,k}$ the 
Fourier coefficient of order $k$ of $\frac{(\chi-1)^\alpha}{(\bar \chi-1)^\alpha}, 
\frac{(\chi\chi_{0}-1)^\alpha}{(\bar \chi \bar \chi_{0}-1)^\alpha}, \frac{(\chi\bar \chi_{0}-1)^\alpha}{(\bar \chi  \chi_{0}-1)^\alpha}$. Clearly $\tilde\gamma_{k} = \frac{\sin (\pi \alpha)}{\pi} \frac{1}{k+\alpha}$
$ \gamma_{1,k} = \chi_{0}^k \tilde\gamma_{k} , \gamma_{2,k} = (\bar\chi_{0})^k \tilde\gamma_{k}. $
Assume also $ k\ge 0$. We have
$ \displaystyle{
\gamma_{-k}= \sum_{v+u=-k} \gamma_{1,u} \gamma_{2,v}}.$
For  an integer $, k_0$ and $k_{0}= k^\tau$, $0<\tau<1$ we can split this sum into
\begin {align*}
&\sum_{u< -k-k_{0}} \gamma_{1,u} \gamma_{2,-k-u} + \sum_{u=-k-k_{0}}^{-k+k_{0}} 
\gamma_{1,u} \gamma_{2,-k-u}
+ \sum_{u=-k+k_{0}+1}^{-k_{0}-1} \gamma_{1,u} \gamma_{2,-k-u} \\
&+\sum_{u= -k_{0}}^{k_{0}} \gamma_{1,u} \gamma_{2,k-u} 
+\sum_{u>k_{0}} \gamma_{1,u} \gamma_{2,-k-u}.
\end{align*}
Write
$$\sum_{u=-k_{0}}^{k_{0}} \gamma_{1,u} \gamma_{2,-k-u} 
= \sum_{u=-k_{0}}^{k_{0}} \gamma_{1,u} 
(\bar \chi_{0})^{k+u}  (\tilde \gamma_{-k-u} -\tilde \gamma_{-k}+\tilde \gamma_{-k}).$$
Since
\begin{equation} \label{UNIF1}
\sum_{u=- k_{0}}^{k_{0}} \gamma_{1,u} (\bar \chi_{0})^{k+u}  (\tilde \gamma_{-k-u} -\tilde \gamma_{-k})
= \frac{\sin (\pi \alpha)}{\pi}\sum_{u=- k_{0}}^{k_{0}} \gamma_{1,u} (\bar \chi_{0})^{k+u} \frac{-u}{(k+u+\alpha)(k+\alpha)}
\end{equation}
it follows that 
\begin{align*}
\sum_{u=-k_{0}}^{k_{0}} \gamma_{1,u} \gamma_{2,-k-u} &= \tilde \gamma_{-k}  \sum_{u=-k_{0}}^{k_{0}}
\gamma_{1,u} (\bar \chi_{0})^{-k-u} + O(k_{0} k^{-2})\\
&= \tilde \gamma_{-k}  (\chi_{0})^k\left( \frac{\chi_{0}^2 -1}{ (\bar \chi_{0})^2 -1} \right)^\alpha\\
&+ \tilde \gamma_{-k}  (\chi_{0})^k \sum_{\vert u\vert \ge k_{0} } \gamma_{1,u} \chi_{0}^{u}
+O(k_{0} k^{-2})\\
&= \tilde \gamma_{-k}  (\chi_{0})^k 
\left(\frac{\chi_{0}^2 -1}
{ (\bar \chi_{0})^2 -1}\right)^\alpha +
O\left((k_{0}k)^{-1}\right) +O(k_{0} k^{-2})\\
&= \tilde \gamma_{-k}  (\chi_{0})^k 
\left(\frac{\chi_{0}^2 -1}{ (\bar \chi_{0})^2 -1}\right)^\alpha +
O(k^{\tau-2}).
\end{align*}
In the same way we have 
$$
\sum_{u=-k -k_{0}}^{-k+k_{0}} \gamma_{1,u} \gamma_{2,k-u} =
\tilde \gamma_{-k}  (\chi_{0})^{-k} 
\left(\frac{(\bar\chi_{0})^2 -1}{ \chi_{0}^2 -1}\right)^\alpha +
O(k^{\tau-2}).
$$
Now using Lemma \ref{APPENDIX1} it is easy to see that 
\begin{equation}\label{UNIF2}
 \sum_{u< -k-k_0} \gamma_{1,u} \gamma_{2,-k-u} \le M_{1} (k_{0}k)^{-1}
 \end{equation}
\begin{equation}\label{UNIF3}
\sum_{u> k_{0}} \gamma_{1,u} \gamma_{2,-k-u} \le M_{2}(k_{0}k)^{-1}
\end{equation}
with $M_{1}$ and $M_{2}$ no depending from $k$.
For the sum  $S=\displaystyle{\sum_{u=-k+k_{0}+1}^{-k_{0}-1} \gamma_{1,u} \gamma_{2,-k-u}} $
we can remark, using an Abel summation, that 
$$\vert S\vert  \le M_{3} (k_{0}k)^{-1}  + \sum_{u=-k+k_{0}+1}^{-k_{0}-1} 
\Bigl \vert \frac{ 1}{(u+\alpha)(k-u+\alpha)} - \frac{1}{(u+1+\alpha)(k-u-1+\alpha)}\Bigr \vert $$
$M_{3}$ no depending from $k$.
Consequently the main values theorem provides
\begin{equation} \label{UNIF4}
 \vert S \vert  \le M_{3}(k_{0}k)^{-1} + \sum_{u=-k+k_{0}+1}^{-k_{0}-1} \frac{k-2u} {(k-u)^2 u^2}.
 \end{equation}
 with $M_{3}$ no depending from $k$.
Then Euler and Mac-Laurin formula provides the upper bound 
$$ \vert S \vert \le O\left( (k_{0}k)^{-1}\right) + \int_{-k+k_{0}+1}^{-k_{0}-1}\frac{k-2u} {(k-u)^2 u^2} du . $$
Since
$$ \int_{-k+k_{0}+1}^{-k_{0}-1}\frac{k-2u} {(k-u)^2 u^2} du \le \frac{3 k}{(k+k_{0})^2} 
\int_{-k+k_{0}+1}^{-k_{0}-1}\frac{1}{u^2} du$$
we get 
$$ \sum_{u<-k+k_{0}+1}^{-k_0-1} \gamma_{1,u} \gamma_{2,k-u} = O\left((k_{0}k)^{-1}\right)$$
and
$$\gamma_{k} =\frac{2}{k+\alpha} \frac { \sin (\alpha)} {\pi} 
\cos \left( \theta_0 k +\phi_{\alpha}\right)+ O\left( (k_{0}k)^{-1}\right) +O(k^{\alpha-2}).$$
Then with a good choice of $\tau$ we obtain the expected formula. The uniformity is a direct consequence of the equations (\ref{UNIF1}), (\ref{UNIF2}), (\ref{UNIF3}), (\ref{UNIF4}).
\end{proof}
The rest of the proof  of Lemma \ref{PRELI2} can be treated as the end of the proof of property \ref{PROP1}.

\subsection {Expression of $\left(T_{N}^{-1} \left( 2^{2\alpha} (\cos \theta-\cos \theta_{0}) ^{2\alpha}c_{1}\right)\right)_{k+1,1}$.}
First we have to prove the next lemma 
\begin{lemma}\label{INVERS3}
For $\alpha \in ]-\frac{1}{2}, \frac{1}{2}[$ we have a function 
$F_{N,\alpha}\in C^1[0 ,\delta ]$ for all $\delta \in ]0,1[$, such that 
\begin{itemize}
\item [i)]
$$Ê\forall z \in [0,\delta [ \quad \vert F_{N,\alpha}(z)\vert \le K_{0}(1+\vert\ln (1-z+\frac{1+\alpha}{N})\vert )$$
where $K_{0}$ is a constant no depending from $N$.
\item [ii)]
$F_N$ and $F'_N$ have a modulus of continuity 
no depending from $N$.
\item [iii)]
\begin{align*}
& \left(T_{N}^{-1} \left(\vert \chi-\chi_{0}\vert ^{2\alpha} \vert \chi-\bar \chi_{0}\vert ^{2\alpha} c_{1}\right)\right)_{k+1,1}
=
\\ &=  \left(\beta^{(\alpha)}_{k,\theta_{0},c_{1}} -\frac{2}{N} \sum_{u=0}^k \beta^{(\alpha)}_{k-u,\theta_{0},c_{1}}  F_{\alpha,N} (\frac{u}{N})
\cos (u\theta_{0}) \right)+R_{N,\alpha}
\end{align*}
uniformly in $k$, $0\le k\le N$, with 
$$R_{N,\alpha} = o\left(N^{-1} \sum_{u=0}^k \beta^{(\alpha)}_{k-u,\theta_{0},c_{1}} F_{\alpha,N} (\frac{u}{N})\right) \quad \mathrm{if} \quad \alpha>0$$
and
$$R_{N,\alpha} = o\left(N^{\alpha-1} \sum_{u=0}^k \beta^{(\alpha)}_{k-u,\theta_{0},c_{1}}  F_{\alpha,N} (\frac{u}{N})\right)\quad \mathrm{if} \quad \alpha<0$$
\end{itemize}
\end{lemma}
\begin{remark} 
 This lemma and the continuity of $F_{N,\alpha}$ in zero imply directly Theorem \ref{COEF2}.
\end{remark}
\begin{remark} \label{REMARQUE2}
Lemma \ref{INVERS3} and the continuity of the function $F_{\alpha}$ imply that 
$$  \left(T_{N}^{-1} \left( 2^{2\alpha} (\cos \theta-\cos \theta_{0}) ^{2\alpha}c_{1}\right)\right)_{1,1}
= \beta^{(\alpha)}_{0,\theta_0,c_{1}} + \frac{1}{N} \beta_{0,\theta_0,c_{1}}^{(\alpha)} F_{N,\alpha} (0) \left(1+o(1)\right).
$$
Since $F_{N,\alpha} (0)= \alpha^2+o(1)$ (see \cite{RS10}) the hypothesis $ \beta_{0,\theta_0,c_{1}} =1$ means that the coefficients of the predictor polynomial are $\left(T_{N}^{-1} \left( 2^{2\alpha} (\cos \theta-\cos \theta_{0}) ^{2\alpha}c_{1}\right)\right)_{k+1,1} \left( 1+o(1)\right)$ uniformly in $k$ 
(it is a direct consequence of the equality (\ref {predizero}).Indeed these of the orthogonal polynomial are 
$$\overline{\left(T_{N}^{-1} \left( 2^{2\alpha} (\cos \theta-\cos \theta_{0}) ^{2\alpha}c_{1}\right)\right)_{N-k+1,1}} \left( 1+o(1)\right)$$ (we can refer to the equations \ref {predizero} and \ref {predi}).
\end{remark}
\begin{proof}{of the lemma \ref{INVERS3}}
In the rest of the paper we slighty change of notation and denote by $\gamma_{k}$ the Fourier coefficient of order $k$ of the function 
$\frac{(\chi\chi_{0}-1)^\alpha (\chi\bar \chi_{0}-1)^\alpha c_{1,1}}{ (\bar \chi \bar \chi_{0}-1)^\alpha
( \bar \chi \chi_{0}-1)^\alpha \bar c_{1,1}}$ by $\gamma_{k}$.
As for \cite{RS10} and using the inversion formula 
and Corollary  \ref{INVERS2} we have to consider the sums 
\begin{align*}
H_{m,N}(u) &= \left (  
\sum_{n_{0}=0}^\infty \gamma_{-(N+1+n_{0}) }
\sum_{n_{1}=0}^\infty \overline{\gamma_{-(N+1+n_{1}+n_{0}) }}
\sum_{n_{2}=0}^\infty \gamma_{-(N+1+n_{1}+n_{2}) }\right.
\times \cdots 
\\
&\times \left.\sum_{n_{2m-1}=0}^\infty \overline{ \gamma_{-(N+1+n_{2m-2}+n_{2m-1}) }}
\sum_{n_{2m}=0}^\infty \gamma_{-(N+1+n_{2m-1}+n_{2m}) }
\overline{\gamma_{u-(N+1+n_{2m})}}\right).
\end{align*}
If 
$$ S_{2m} = \sum_{n_{2m}=0}^\infty \gamma_{-(N+1+n_{2m-1}+n_{2m}) }
\overline{\gamma_{u-(N+1+n_{2m})}}$$
 we can write, following the previous Lemma, 
 $ S_{2m} = S_{2m,0}+ S_{2m,1} $ with 
 \begin{align*}
 S_{2m,0} &=  4  \left( \frac{\sin (\pi \alpha)}{\pi}\right)^2 \\
 & \sum_{n_{2m}=0}^\infty \cos \left((N+1+n_{2m-1}+n_{2m})\theta_{0}+2\omega'_{\alpha,\theta_{0}}\right)
  \cos \left((N+1+n_{2m})-u )\theta_{0}+2\omega'_{\alpha,\theta_{0}}\right)\\
&\times  \frac {1} {N+1+n_{2m-1}+n_{2m}+\alpha}\frac{1}{N+1+n_{2m}-u+\alpha}\\
&=  2  \left( \frac{\sin (\pi \alpha)}{\pi}\right)^2 \left( \sum_{n_{2m}=0}^\infty
\frac{\cos \left( n_{2m-1}+u )\theta_{0}\right)} {N+1+n_{2m-1}+n_{2m}+\alpha}\frac{1}{N+1+n_{2m}-u+\alpha}\right. \\
&+ 
   \sum_{n_{2m}=0}^\infty
\cos \left( \left(2 (N+1+n_{2m}+4\omega'_{\alpha})+n_{2m-1}-u)\right)\theta_{0}\right)\\
&\times \left.   \frac {1} {N+1+n_{2m-1}+n_{2m}+\alpha}\frac{1}{N+1+wn_{2m}-u+\alpha}\right)
  \end{align*}
  Let us study the order of the second sum.
To do this we can evaluate the order of the expression
$$ \sum_{j=0}^M \chi_{0}^{j} \frac {1} {N+1+n_{2m-1}+j+\alpha}\frac{1}{N+1+j-u+\alpha} $$
where $M$ goes to the infinity and $N=o(M)$.
As for the previous proofs it is clear that this sum is bounded by 
$$\sum_{j=0}^M\Bigl \vert \frac{1} {N+2+n_{2m-1}+j}\frac{1}{N+2+j-u} 
-\frac{1} {N+1+n_{2m-1}+j}\frac{1}{N+1+j-u}.\Bigr \vert $$
Obviously
\begin{align*}
&\Bigl \vert \frac{1} {N+2+n_{2m-1}+j}\frac{1}{N+2+j-u} 
-\frac{1} {N+1+n_{2m-1}+j}\frac{1}{N+1+j-u}\Bigr \vert  \\
&\le \Bigl \vert \frac{2N+2 +2 j+n_{2m-1}-u} { (N+1+n_{2m-1}+j)^2(N+1+j-u)^2} \Bigr \vert 
\end{align*}
and 
\begin{align*}
&\Bigl \vert \frac{2N+2+2 j+n_{2m-1}-u} { (N+1+n_{2m-1}+j)^2(N+1+j-u)^2} \Bigr \vert\\
&= \Bigl \vert \frac{1}{N+1+j +n_{2m-1}}+\frac{1}{N+1+j -u}\Bigr \vert 
\frac{1}{(N+1+j +n_{2m-1}) (N+1+j -u)} \\
&\le  \frac{1}{N} \frac{1}{(N+1+j +n_{2m-1}) (N+1+j -u)}.
\end{align*}
In the other hand we have, for $\alpha \in ]0, \frac{1}{2}[$ 
$$ S_{2m,1} = o\left( \sum_{n_{2m}=0}^\infty \frac{1}{N+1+n_{2m-1}+n_{2m}+\alpha}
\frac{1}{N+1+n_{2m}-u+\alpha}\right)$$
and for $\alpha\in]-\frac{1}{2},0[.$
$$ 
S_{2m,1} = o\left(N^\alpha \sum_{n_{2m}=0}^\infty \frac{1}{N+1+n_{2m-1}+n_{2m}+\alpha}
\frac{1}{N+1+n_{2m}-u+\alpha}\right).
$$

Hence we can write 
 $$ S_{2m}= S'_{2m}\left( \cos \left(\theta_{0}(n_{2m-1}+ u) \right) +r_{m,\alpha}\right),$$
 with 
 $$ S'_{2m}= \sum_{n_{2m}=0}^{+ \infty} \frac{1} {N+1+n_{2m-1}+n_{2m}+\alpha}\frac{1}{N+1+n_{2m}-u+\alpha}.$$
 and 
 \[ \left\{ 
\begin{array}{cccc}
  r_{m,\alpha}  =  &  o(1)            &\mathrm{if} &\quad  \alpha\in]0, \frac{1}{2}[\\
   r_{m,\alpha}  =  &  o(N^\alpha)&\mathrm{if} &\quad \alpha\in]- \frac{1}{2},0[. 
   \end{array}
\right.
\]
  For  $z\in [0,1]$ we define $F_{m,N,\alpha}(z)$ by 
\begin{align*}
F_{m,N,\alpha}(z) =& \sum_{n_{0}=0}^\infty \frac{1}{N+1+n_{0}} \sum_{n_{1}=0}^{\infty} \frac{1}{N+1+w_{1}+w_{0}}
\times \cdots \\
\times & \sum_{n_{2m-1}=0}^\infty \frac{1}{N+1+n_{2m-2}+n_{2m-1}+\alpha} \\
\times& \sum_{n_{2m}=0}^\infty \frac{1}{N+1+n_{2m-1}+n_{2m}+\alpha}
\frac{1}{1+\frac{1+\alpha}{N}+\frac{n_{2m}}{N}-z}.
\end{align*}

 Repeating the same  idea as previously for the sums on $n_{2m-1}, \cdots, n_{0}$ we finally obtain
$$H_{m,N} (u) = \frac{2}{N} \left( \frac{\sin(\pi \alpha)}{\pi} \right)^{2m+2} F_{m,N,\alpha} (\frac{u}{N})
(\cos(u\theta_{0}) +R_{N,\alpha}).
$$
with $R_{N,\alpha}$ as announced previously.\\
  We established in \cite{RS10} the continuity of the function $F_{m,N, \alpha}$ and  the 
uniform convergence in $[0,1]$ of the sequence 
$\displaystyle{\sum_{m=0}^\infty 
\left(\frac{\sin(\pi \alpha)}{\pi}\right)^{2m} F_{m,N,\alpha}(z)}$.
Let us denote by $F_{N,\alpha}  (z)$ the sum
$\displaystyle{ \sum_{m=0}^{+\infty} \left( \frac{\sin \pi \alpha}{\pi}\right) ^{2 m} F_{m,N,\alpha}(z)}$.
The function $F_{N,\alpha}$ is defined, continuous and derivable on $[0,1[$ (see \cite{RS10} Lemma 4). Moreover for all $z\in [0,\delta ]$ , $0<\delta <1$ we have the inequality 
$$ \frac{1}{1+\frac{1+\alpha}{N} +\frac{n_{2m}}{N} -z} \le \frac{1}{1+\frac{1+\alpha}{N}  -\delta }.$$ 
Hence 
$$\left( \frac{1+\frac{1+\alpha}{N}-\delta } {1+\frac{1+\alpha}{N} +\frac{n_{2m}}{N} -z} \right)^2 \le
 \frac{1+\frac{1+\alpha}{N}-\delta }{1+\frac{1+\alpha}{N} +\frac{n_{2m}}{N} -z }$$ 
and 
$$\left(\frac{1 } {1+\frac{1+\alpha}{N} +\frac{n_{2m}}{N} -z} \right)^2 \le
 \frac{1} {1+\frac{1+\alpha}{N}-\delta } \frac{1}{1+\frac{1+\alpha}{N} +\frac{n_{2m}}{N} -z}.$$
These last inequalities and the proof of Lemma 4 in \cite{RS10} prove that
$F_{N,\alpha}$ is in $C^1[0,1[$.

Always in \cite{RS10} we have obtained that, for all $z$ in $[0,1]$,
\begin{equation}\label{F}
\Bigl \vert   F_{N,\alpha}(z) \Bigr \vert \le K_{0} \left( 1+ \Bigr \vert \ln ( 1-z+\frac{1+\alpha}{N})\Bigl \vert \right)
\end{equation}
where $K_{0}$ is a constant no depending from $N$.\\
 Now we have to prove the point ii) of the statement. For $z,z' \in [0,\delta]$ 
 \begin{align*}
 &\Bigl \vert \frac{ z-z'}{ (1+\frac{1+\alpha}{N} +
 \frac{n_{2m}}{N} -z) (1+\frac{1+\alpha}{N} +
 \frac{n_{2m}}{N} -z')}\Bigr\vert \\
& \le  \frac{\vert z-z'\vert }{ 1-\delta} 
\frac{1}{ 1+\frac{1+\alpha}{N} +
 \frac{n_{2m}}{N} -\delta}
 \end{align*}
 that implies, with the inequality (\ref{F})
\begin{equation}\label{unifcont1}
\vert F_{N,\alpha} (z) -F_{N,\alpha} (z')Ê\vert 
 \le \vert z-z'\vert \frac{ K_0 \left ( 1+ \Bigl \vert
 \ln (1-\delta +\frac{1+\alpha}{N})\Bigr \vert \right)}
 {1-\delta}.
 \end{equation}
 In the same way we have 
 \begin{align*}
 &\vert z-z'\vert \Bigl \vert \frac{ ((1+\frac{1+\alpha}{N} +
 \frac{n_{2m}}{N} -z)+ (1+\frac{1+\alpha}{N} +
 \frac{n_{2m}}{N} -z')}{ (1+\frac{1+\alpha}{N} +
 \frac{n_{2m}}{N} -z)^2 (1+\frac{1+\alpha}{N} +
 \frac{n_{2m}}{N} -z')^2}\Bigr\vert \\
& \le 2 \vert z-z'\vert 
  \frac{1}{ (1-\delta)^2}  
\frac{1}{ 1+\frac{1+\alpha}{N} +
 \frac{n_{2m}}{N} -\delta}
 \end{align*}
and 
 always with the inequality (\ref{F})
\begin{equation}\label{unifcont2}
\vert F'_{N,\alpha} (z) -F'_{N,\alpha} (z')Ê\vert 
 \le2 \vert z-z'\vert \frac{ K_0 \left ( 1+ \Bigl \vert
 \ln (1-\delta +\frac{1+\alpha}{N})\Bigr \vert \right)}
 {(1-\delta)^2}.
 \end{equation}
Using  (\ref{unifcont1}) and (\ref{unifcont2})
we get the point $ii)$.\\
 To achieve the proof we have  to remark that the uniformity in $k$ in the point $iii)$ is a direct consequence of Property 2. 
\end{proof}
We have now to state the following lemma  
\begin{lemma}\label{final}
for $\frac{k}{N}\rightarrow x$, $0<x<1$ we have 
$$2 \sum_{u=0}^k \overline{ \beta_{k-u,\theta_{0},c_{1}}^{(\alpha)}} \cos (u \theta_{0}) 
F_{N,\alpha}(\frac{u}{N})
= K_{\alpha,\theta_{0},c_{1}} \cos (k\theta_{0}+\omega_{\alpha,\theta_0}) 
\sum_{u=0}^k  \tilde \beta_{k-u}^{(\alpha)} F_{N\alpha} (\frac{u}{N})+o(k^{\alpha-1}),$$
uniformly in $k$ for $x$ in all compact of $]0,1[$ 
\end{lemma}
\begin{remark}
This Lemma and Lemma \ref{INVERS3}  imply  the equality
\begin{align*}
T_{N} ^{-1}& \left( \vert \chi - \chi_{0}\vert^{2\alpha} \vert \chi - \bar \chi_{0} \vert ^{2\alpha} c_{1} \right)_{k+1,1} =\\
&K_{\alpha,\theta_{0},c_{1}}\cos (k\theta_{0}+\omega_{\alpha,\theta_{0}})
T_{N} ^{-1} \left( \vert 1-\chi\vert^{2\alpha}\right)_{k+1,1} +o(k^{\alpha-1})
\end{align*}
with (see \cite{RS10} Lemma 3)
$$T_{N} ^{-1} \left( \vert 1-\chi\vert ^{2\alpha} \right)_{k+1,1}= \left( {\tilde\beta_{k}} ^{(\alpha)} - \frac{1}{N} \sum_{u=0}^k {\tilde \beta}_{k-u}^{(\alpha)} F_{N,\alpha}(\frac{u}{N})\right).
$$
\end{remark}
\begin{proof}{of lemma \ref{final}}
With our notation assume $x \in [0,\delta]$, $0<\delta <1$.
Put 
$k_0 =N^\gamma$ with $\gamma\in ]\max(\frac{\alpha}{\beta}, \frac{-\alpha}{1-\alpha}),1[$ if $\alpha<0$, and 
 $\gamma\in ]0,1[$ if $\alpha>0$. We can splite the sum 
$\displaystyle{
2  \sum_{u=0}^k  \beta_{k-u,\theta_{0},c_{1}}^{(\alpha)} F_{N,\alpha} (\frac{u}{N})\cos (u \theta_{0}) }$ into 
$\displaystyle{2 \sum_{u=k-k_0}^{k}  \beta_{k-u,\theta_{0},c_{1}}^{(\alpha)} F_{N,\alpha} (\frac{u}{N}) \cos (u \theta_{0}) }$
and 
$2 \displaystyle{\sum_{u=0}^{k-k_0}  \beta_{k-u,\theta_{0},c_{1}}^{(\alpha)} F_{N,\alpha} (\frac{u}{N}) \cos (u \theta_{0}). }$
Property \ref{PROP1} and the assumption on  $\beta$ show that 
 \begin{align*}
 2 \sum_{u=0}^{k-k_0} \overline{ \beta_{k-u,\theta_{0},c_{1}}^{(\alpha)}} F_{N,\alpha} (\frac{u}{N}) \cos (u \theta_{0})&= 
  2 K_{\alpha,\theta_{0},c_{1}} \\
 &\times \sum_{u=0}^{k-k_0}  {\tilde\beta_{k-u}}^{(\alpha)} \cos ((k-u)\theta_{0}+\omega_{\alpha,\theta_0}) \cos (u \theta_{0})F_{N,\alpha}(\frac{u}{N}) +o(k^\alpha)
  \\
  &= K_{\alpha,\theta_{0},c_{1}} \left(
   \sum_{u=0}^{k-k_0} {\tilde\beta_{k-u}}^{(\alpha)} \cos (k\theta_{0}+\omega_{\alpha,\theta_0}) F_\alpha(\frac{u}{N})\right.\\
  &+ \left.
 \sum_{u=0}^{k-k_0} {\tilde\beta_{k-u}}^{(\alpha)} \cos ((k-2u)\theta_{0})+\omega_{\alpha,\theta_0}) F_{N,\alpha} (\frac{u}{N}) \right) +o(k^{\alpha}),
\end{align*}
uniformly in $k$.
It is known that the second sum is also 
$$  \sum_{u=0}^{k-k_0} \frac{(k-u)^{\alpha-1}}{\Gamma(\alpha)} \cos ((k-2u)\theta_{0})+\omega_{\alpha,\theta_0}) F_{N,\alpha} (\frac{u}{N})+o(k^{\alpha-1}),$$
uniformly in $k$ with the equation (\ref{ZYG}).
Then an Abel summation provides that the quantity\\
$ \Bigr\vert\displaystyle{\sum_{u=0}^{k-k_0} (k-u)^{\alpha-1} \cos ((k-2u)\theta_{0}+\omega_{\alpha,\phi_0}) F_{N,\alpha} (\frac{u}{N})}\Bigl \vert  $
is bounded by \\
$ M_1 k_{0}^{\alpha-1}+\displaystyle{\sum_{u=0}^{k-k_0} 
\vert (k-u-1)^{\alpha-1} F_{N,\alpha} (\frac{u+1}{N})
- (k-u)^{\alpha-1} F_{N,\alpha} (\frac{u}{N})\vert}$
with $M_1$ no depending from $k$. 
Moreover 
\begin{align*} 
&\sum_{u=0}^{k-k_0} 
\vert (k-u-1)^{\alpha-1} F_{N,\alpha} (\frac{u+1}{N})
- (k-u)^{\alpha-1} F_{N,\alpha} (\frac{u}{N})\vert \\
&\le 
\sum_{u=0}^{k -k_0}
\vert (k-u-1)^{\alpha-1} - (k-u)^{\alpha-1} \vert 
\vert F_{N,\alpha}(\frac{u}{N})\vert\\
&+\sum_{u=0}^{k-k_0}
\vert  F_{N,\alpha} (\frac{u+1}{N})
-  F_{N,\alpha}(\frac{u}{N})\vert  \vert(k-u-1)^{\alpha-1}\vert 
\end{align*}
From the inequality \ref{F} (we have assumed $0<\frac{k}{N}<\delta$) we infer
$$ \sum_{u=0}^{k-k_0} 
\vert (k-u-1)^{\alpha-1} - (k-u)^{\alpha-1} \vert 
\vert F_{N,\alpha} (u)\vert\le M_2\sum_{w=k_{0}}^{k}
 v^{\alpha-2}$$
 with $M_{2}$ no depending from $k$. We finally get 
\begin{align*}
\sum_{u=0}^{k-k_0} 
\vert (k-u-1)^{\alpha-1} - (k-u)^{\alpha-1} \vert 
\vert F_{N,\alpha} (u)\vert &= O\left( \sum_{w=k_{0}}^{k}
 v^{\alpha-2}\right) \\
&= O \left( k_0 ^{\alpha-1}\right) =o(k^{\alpha})
 \end{align*} 
 Identically  Lemma \ref{INVERS3} and the main value 
 theorem provides
 $$ 
 \sum_{u=0}^{k-k_0}
\vert  F_{N,\alpha} (\frac{u+1}{N})
-  F_{N,\alpha}(\frac{u}{N})\vert  \vert(k-u-1)^{\alpha-1}\vert \le M_3 \frac{k^\alpha}{N} =
o(k^{\alpha})$$
  with $M_3$ no depending from $N$.
By definition of $k_0$ and with Property \ref{PROP1} we have easily the existence of a constant $M_{4}$, always no depending from $k$,
such that for $\alpha>0$ 
$$\Bigl \vert\sum_{u=k-k_0}^{k}  \beta_{k-u,\theta_{0},c_{1}}^{(\alpha)} F_{N,\alpha}(\frac{u}{N}) 
\cos (u \theta_{0}) \Bigl \vert
\le M_{4} k_{0}^{\alpha}.$$
 Consequently for $\alpha>0$ 
\begin{align*}
 2 \sum_{u=0}^{k}  \beta_{k-u,\theta_{0},c_{1}}^{(\alpha)}& F_{N,\alpha} ( \frac{u}{N}) \cos (u \theta_{0}) 
\\& =K_{\alpha,\theta_{0},c_{1}} \cos ((k-\alpha)\theta_{0}+\omega_{\alpha,\theta_{0}})
 \sum_{u=0}^k \tilde  \beta_{k-u}^{(\alpha)}  F_{N,\alpha} (\frac{u}{N}) +o(k^{\alpha}).
 \end{align*}
 uniformly in $k$ with the definition of the constants
 $M_i$, $1\le i \le 4$
and we get the Lemma for $\alpha>0$.\\
Since we have the  result for the positive case we assume in the rest of 
the demonstration that  $\alpha \in ]-\frac{1}{2},0[.$ Recall that now 
$\gamma\in ]\max (\frac{\alpha}{\beta},
\frac{-\alpha}{1-\alpha}),1[$.\\
First we have to evaluate the sum 
$ \displaystyle{ \sum_{u=k-k_{0}} ^k\overline{\beta^{(\alpha)}_{k-u,\theta_{0},c_{1}} }\cos ( u\theta_{0}) F_{\alpha}(u). }$
Since $F_{N,\alpha} \in C^1[0, \delta ]$ we have for $\frac{k-k_{0}}{N}\le \frac{u}{N} \le \frac{k}{N} 
\le \delta < $ the formula
$ F_{\alpha,N} (\frac{u}{N})  - F_{\alpha,N} (\frac{k}{N}) + F_{\alpha,N} (\frac{k}{N}) =
 F_{\alpha,N} (\frac{k}{N})  + O(\frac{k_{0}}{N}) = 
 F_{\alpha,N} (\frac{k}{N})+o(k^\alpha)$ uniformly in $k$ (see once a more the definition of $\gamma$).\\
Property \ref{PROP1} provides   
$ \overline{ \beta_{k-u,\theta_{0},c_{1}} ^{(\alpha)} } =  \beta_{k-u,\theta_{0},c_{1}} ^{(\alpha)}+ o(k^{\beta-1})$. 
Hence we can write, uniformly in $k$, 
\begin{align*}
2\sum_{u=k-k_{0}} ^k \beta^{(\alpha)}_{k-u,\theta_{0},c_{1}} \cos ( u\theta_{0}) F_{\alpha}(\frac{u}{N})
&= 2\Re \left( \chi_{0}^k \sum_{u=k-k_{0}} ^k \beta^{(\alpha)}_{k-u,\theta_{0},c_{1}} 
(\bar\chi_{0})^{k-u} F_{\alpha}(\frac{k}{N})\right)  +o(k^\alpha)
\\
&= 2 \Re \left( \chi_{0}^k \sum_{v=0} ^{k_{0}} \beta^{(\alpha)}_{v,\theta_{0},c_{1}} 
(\bar\chi_{0})^{v} F_{\alpha}(\frac{k}{N})\right)+o(k^\alpha)\\
&= -2\Re \left( \chi_{0}^k \sum_{v=k_{0}+1} ^\infty \beta^{(\alpha)}_{v,\theta_{0},c_{1}} 
(\bar\chi_{0})^{v} F_{\alpha}(\frac{k}{N})\right) +o(k^\alpha) . 
\end{align*}
Moreover we have, uniformly with Property \ref{PROP1}, 
$$ 2\sum_{v=k_{0}+1} ^\infty \beta^{(\alpha)}_{v,\theta_{0},c_{1}} (\bar\chi_{0})^{v}
 = K_{\alpha,\theta_{0},c_{1}}  \sum_{v=k_{0}+1} ^\infty {\tilde\beta}^{(\alpha)}_{v}
  \left( e^{i (v\theta_{0}+\omega_{\alpha,\theta_{0}})} 
  +  e^{-i(v\theta_{0}+\omega_{\alpha,\theta_{0}})}\right) e^{-iv\theta_{0}} + o (k_{0}^\beta).$$
  Consequently $\gamma\in]\max (\frac{\alpha}{\beta},\frac{-\alpha}{1-\alpha}),1[$ infer that 
  $$ 2\sum_{v=k_{0}+1} ^\infty \beta^{(\alpha)}_{v,\theta_{0},c_{1}} (\bar\chi_{0})^{v}
 = K_{\alpha,\theta_{0},c_{1}}  \sum_{v=k_{0}+1} ^\infty {\tilde\beta}^{(\alpha)}_{v}
  \left(e^{i(v\theta_{0}+\omega_{\alpha,\theta_{0}})} 
  + e ^{-i(v\theta_{0}+\omega_{\alpha,\theta_{0}})}\right)e^{-iv\theta_{0}} + o (k^\alpha).$$
We have 
\begin{align*}
& \sum_{v=k_{0}+1} ^\infty {\tilde\beta}^{(\alpha)}_{v} ( e^{i(v\theta_{0}+
\omega_{\alpha,\theta_0})} + e^{-i(v \theta_{0}+\omega_{\alpha,\theta_0})})e^{-iv\theta_{0}} \\
& =
\sum_{v=k_{0}+1} ^\infty {\tilde \beta}^{(\alpha)}_{v} ( e^{i(\omega_{\alpha,\theta_{0}})
}+ e^{-i(2v \theta_{0} + \omega_{\alpha,\theta_0})})
\\
&= \sum_{v=k_{0}+1} ^\infty {\tilde \beta}^{(\alpha)}_{v}  e^{i(\omega_{\alpha,\theta_{0}})} 
+ R.
 \end{align*}
 An Abdel summation provides $\vert R\vert \le M_{4} k_{0}^{\alpha-1} =o(k^\alpha)$ uniformly in 
 $k$.
 
Hence we have 
\begin{align*}
2\sum_{u=k-k_{0}} ^k \beta^{(\alpha)}_{k-u,\theta_{0},c_{1}} \cos ( u\theta_{0}) F_{\alpha}(\frac{u}{N})&=
-K_{\alpha,\theta_{0},c_{1}} \cos \left( k \theta_{0}+\omega_{\alpha,\theta_0}\right)
 \sum_{v=k_{0}+1} ^\infty{\tilde \beta}^{(\alpha)}_{v} 
F_{N,\alpha}(\frac{k}{N}) +o(k^{\alpha})\\
&= K_{\alpha,\theta_{0},c_{1}}\cos \left( k \theta_{0}+\omega_{\alpha,\theta_0}\right)
\sum_{v=0} ^{k_{0}}{\tilde \beta}^{(\alpha)}_{v} F_{N,\alpha}(\frac{k}{N}) +o(k^\alpha).
\end{align*}
With Lemma \ref{INVERS3} we obtain, as previously   
$$ \sum_{u=k-k_{0}} ^k {\tilde \beta}^{(\alpha)}_{k-u}  F_{\alpha}(\frac{u}{N}) 
= \sum_{v=0} ^{k_{0}}{\tilde\beta}^{(\alpha)}_{v} 
F_{N,\alpha}(\frac{k}{N})+o(k^{\alpha})$$
uniformly in $k$.
Since we have seen that the sum
$$
2\sum_{u=0}^{k-k_{0}} \beta_{k-u,\theta_{0},c_{1}}^{(\alpha)} F_{N,\alpha}(\frac{u}{N}) \cos (u \theta_{0})
$$
is equal to 
$$
K_{\alpha,\theta_{0},c_{1}}\cos (k\theta_{0}+\omega_{\alpha,\theta_0})
  \sum_{u=0}^{k-k_0} {\tilde\beta_{k-u}}^{(\alpha)}  F_{N,\alpha} (\frac{u}{N}) +o(k^{\alpha})$$
 we can also conclude, as for $\alpha>0$ 
\begin{align*}
 2 \sum_{u=0}^{k} \beta_{k-u,\theta_{0},c_{1}}^{(\alpha)} &F_{N,\alpha}( \frac{u}{N}) \cos (u \theta_{0}) 
 \\&= K_{\alpha,\theta_{0},c_{1}} \cos (k\theta_{0}
 +\omega_{\alpha,\theta_0})
 \sum_{u=0}^k {\tilde \beta_{k-u}}^{(\alpha)}  F_{N,\alpha}(\frac{u}{N}) +o(k^\alpha).
 \end{align*}
 The uniformity is clearly provided by the uniformity in Lemma \ref{INVERS3} and by the previous 
 remarks.
 This last remark is sufficient to prove Lemma \ref{final}.
\end{proof}
Then Theorem \ref {COEF} is a direct consequence of the inversion formula and of Lemma 
\ref{final}.
\section{Proof of Theorem \ref{TOEP1}}
Let us recall the following formula, which can be related with the Gobberg-Semencul formula.
\begin{lemma}\label{GS}
If $P= \displaystyle{ \sum_{u=0}^N \delta _{u} \chi^u}$ a trigonometric polynomial of degree $N$.
Then we have, if $k\le l$ 
$$\left(T^{-1}_{N} \left( \frac{1}{\vert P \vert ^2}\right) \right)_{k+1,l+1}= \sum_{u=0}^k \bar \delta_{u} \delta_{l-k+u}
- \sum_{u=0}^k \delta _{N-k+u} \bar \delta_{N-l+u}.$$
\end{lemma}
Let $P_{N,\alpha,\theta_{0}}$ and $P_{N,\alpha}$ be the predictor polynomials of 
$\vert \chi -\chi_0\vert^{2\alpha}
\vert \chi \bar \chi_0\vert  c_{1}$ and $\vert1-\chi\vert^{2\alpha} $.  We put 
$ P_{N,\alpha,\theta_{0}} = \displaystyle{ \sum_{u=0}^N \delta_{u,\theta_{0}}^{(\alpha)} \chi^{u}}$ 
and 
$ P_{N,\alpha} = \displaystyle{ \sum_{u=0}^N {\tilde\delta}_{u}^{(\alpha)} \chi^{u}}$. Following Formula 
(\ref{predizero}) we have 
$$
 \delta_{u,\theta_{0}}^{(\alpha)} = 
\frac{T_N^{-1} \left(2^{2\alpha}( \cos \theta-\cos \theta_{0}) ^\alpha c_{1}\right)_{u+1,1} }
{\sqrt{ T_N^{-1} \left(2^{2\alpha}( \cos \theta-\cos \theta_{0}) ^\alpha c_{1}\right)_{1,1}}}$$
and 
$$\tilde  \delta_{u}^{(\alpha)}=
\frac{T_N^{-1} \left(1-\cos \theta) ^\alpha c_{1}\right)
_{u+1,1} }{\sqrt{T_N^{-1} \left(1-\cos \theta) ^\alpha c_{1}\right)_{1,1}}}.$$

Then Remark 
\ref{REMARQUE2} and the hypothesis 
$\beta^{0} _{0,\theta_0,c_1}$ give the equalities 
$$ \delta_{u,\theta_{0}}^{(\alpha)} = 
T_N^{-1} \left(2^{2\alpha}( \cos \theta-\cos \theta_{0}) ^\alpha c_{1}\right)_{u+1,1} \left(1+o(1)\right)$$
and for the same reasons
$$\tilde  \delta_{u}^{(\alpha)}=
T_N^{-1} \left(1-\cos \theta) ^\alpha c_{1}\right)
_{u+1,1} \left(1+o(1)\right).$$
According to Lemma \ref{GS} we have to treat the two sums (with the hypothesis $x <y$) 
$S_{1,\alpha}=\displaystyle{ \sum_{u=0}^k \overline{ \delta_{u,\theta_{0}}^{(\alpha)}}
 \delta_{l-k+u,\theta_{0}}^{(\alpha)}}$ 
and
$ S_{2,\alpha}=\displaystyle{\sum_{u=0}^k \delta _{N-k+u,\theta_{0}}^{(\alpha)} 
\overline{ \delta_{N-l+u,\theta_{0}}^{(\alpha)}}} $.
For a sufficiently large integer $k_{0}$ we can split the sum $S_{1,\alpha}$ into   
$  \displaystyle{ \sum_{u=0}^{k_{0}} \overline{\delta_{u,\theta_{0}}^{(\alpha)}}
 \delta_{l-k+u,\theta_{0}}^{(\alpha)}}$ and $\displaystyle{\sum_{u=k_{0}+1}^k \overline{ \delta_{u,\theta_{0}}^{(\alpha)}}
 \delta_{l-k+u,\theta_{0}}^{(\alpha)}}$.
We have 
 $$
\sum_{u=0}^{k_{0}} \overline{\delta_{u,\theta_{0}}^{(\alpha)}}
 \delta_{l-k+u,\theta_{0}}^{(\alpha)} \le \frac{K_{\alpha,\theta_{0},c_{1}}}{\Gamma(\alpha)} (l-k)^{\alpha-1} (1-\frac{l-k}{N})^{\alpha}
M k_{0}.$$
with $M = \max\{ \delta_{u,\theta_{0}}^{(\alpha)}\}$.
Assume now $k_{0}= N^\gamma$ with $0<\gamma<\alpha$. We get
\begin{equation} \label{PRIMO}
\sum_{u=0}^{k_{0}} \overline{\delta_{u,\theta_{0}}^{(\alpha)}}
 \delta_{l-k+u,\theta_{0}}^{(\alpha)} =o(N^{2\alpha-1})
 \end{equation}
 In the other hand we have, following Theorem \ref{COEF} 
\begin{align*}  
 \sum_{u=k_{0}+1}^k \overline{ \delta_{u,\theta_{0}}^{(\alpha)}}
 \delta_{l-k+u,\theta_{0}}^{(\alpha)} &= \vert  K_{\alpha,\theta_{0}}\vert^2
  \sum_{u=k_{0}+1}^k \cos \left( u\theta_{0}+\omega_{\alpha,\theta_{0}}\right) 
\cos \left( (l-k)\theta_{0}+\omega_{\alpha,\theta_{0}}\right)
\\
& u^{\alpha-1} (1-\frac{u}{N})^\alpha (l-k+u)^{\alpha-1} (1-\frac{l-k+u}{N})^\alpha +o(k^{2\alpha-1})
\end{align*}
As previously we obtain, with an Abel summation, that 
\begin{align*} 
& \sum_{u=k_{0}+1}^k \overline{\delta_{u,\theta_{0}}^{(\alpha)}}
 \delta_{l-k+u,\theta_{0}}^{(\alpha)} \\
 & = \frac{\vert K_{\alpha,\theta_{0},c_{1}} \vert^2}{\Gamma^2(\alpha)}
\cos\left( (l-k) \theta_{0}\right) \sum_{u=k_{0}+1}^k u^{\alpha-1} (1-\frac{u}{N})^\alpha (l-k+u)^{\alpha-1} (1-\frac{l-k+u}{N})^\alpha +S'_{1,\alpha}.
\end{align*}
with 
$$ 
\vert S'_{1,\alpha}\vert = O \left( \sum_{u=k_{0}+1}^{k} \vert \rho_{N }(u+1) -\rho_{N} (u)\vert \right) 
$$
and 
$$ \rho_{N} (u) = u^{\alpha-1} (1-\frac{u}{N})^\alpha (l-k+u)^{\alpha-1} (1-\frac{l-k+u}{N})^\alpha.$$
With the main value theorem we can write 
$$
\vert S'_{1,\alpha}\vert =O \left( \sum_{u=k_{0}+1}^{k} \vert \rho'_{N }(c) \vert \right)
\quad u<c<u+1.
$$
Hence 
$$\vert S'_{1,\alpha}\vert = O \left(\sum_{j=0}^4 \Sigma^{(j)}_{1,\alpha}\right)=o(N^{2\alpha-1}).$$
Finally we obtain
\begin{align*}
& \sum_{u=k_{0}+1}^k \overline{\delta_{u,\theta_{0}}^{(\alpha)}}
 \delta_{l-k+u,\theta_{0}}^{(\alpha)} \\
 & = \frac{\vert K_{\alpha,\theta_{0},c_{1}} \vert^2}{\Gamma^2(\alpha)}
\cos\left( (l-k) \theta_{0}\right) \sum_{u=k_{0}+1}^k u^{\alpha-1} (1-\frac{u}{N})^\alpha (l-k+u)^{\alpha-1} (1-\frac{l-k+u}{N})^\alpha +o(N^{2\alpha-1}).
\end{align*}
As for the equation (\ref{PRIMO}) we get 
$\displaystyle{\sum_{u=0}^{k_0} {\tilde \delta_u}^{(\alpha)} 
{\tilde \delta_{l-k+u}}^{(\alpha)}.}$ 
Consequently we can conclude 
 \begin{equation} \label{DEUXIO}
 S_{1,\alpha} = \vert K_{\alpha,\theta_{0},c_{1}} \vert^2 
  \cos\left( (l-k) \theta_{0}\right)\sum_{u=0}^k \overline{{\tilde\delta}_{u}^{(\alpha)}}
 {\tilde\delta_{l-k+u}}^{(\alpha)} +o(N^{2\alpha-1}).
 \end{equation}
As previously we can split the sum $S_{2,\alpha}$ into
$\displaystyle{ \sum_{u=0}^{k-k_{1}-1} \delta _{N-k+u,\theta_{0}}^{(\alpha)} 
\overline{ \delta_{N-l+u,\theta_{0}}^{(\alpha)}}}$ and 
$ \displaystyle{\sum_{u=k-k_{1}}^k \delta _{N-k+u,\theta_{0}}^{(\alpha)} 
\overline{ \delta_{N-l+u,\theta_{0}}^{(\alpha)}}}.$
Using Lemma \ref{final} we obtain the bound 
\begin{align*}
\Bigl\vert \sum_{u=k-k_{1}}^k \delta _{N-k+u,\theta_{0}}^{(\alpha)} 
\overline{ \delta_{N-l+u,\theta_{0}}^{(\alpha)}} \Bigr \vert & \le 
\sum_{u=k-k_{1}}^k\vert  \beta_{N-k+u,\theta_{0},c_{1}}^{(\alpha)}  \delta_{N-l+u,\theta_{0}}^{(\alpha)}\vert 
\\
&+ \sum_{u=k-k_{1}}^k \vert \delta_{N-l+u,\theta_{0}}^{(\alpha)}\vert 
\frac{1}{N} \sum_{v=0}^{ N-k+u} \vert \beta_{N-k+u-v,\theta_{0},c_{1}}^{(\alpha)}\vert 
\vert F_{N,\alpha}( \frac{v}{N})\vert\\
&+o(N^{2\alpha-1}).
\end{align*}
Assume now $k_{1}=o(N)$. We have 
\begin{align*}
 \sum_{u=k-k_{1}}^k\vert  \beta_{N-k+u,\theta_{0},c_{1}}^{(\alpha)}  \delta_{N-l+u,\theta_{0}}^{(\alpha)}\vert 
&\le  O\left( (N-l+k)^{\alpha-1} (\frac{l-k}{N})^\alpha \sum_{u=k-k_{1}}^k (N-k+u)^{\alpha-1}\right) 
\\ &\le O\left( N^{2\alpha-1} \left( 1- (1-\frac{k_{1}}{N})^\alpha\right)\right) = o(N^{2\alpha-1})
\end{align*}
and
\begin{align*}
&\sum_{u=k-k_{1}}^k \vert \delta_{N-l+u,\theta_{0}}^{(\alpha)}\vert 
\frac{1}{N} \sum_{v=0}^{ N-k+u} \vert \beta_{N-k+u-v,\theta_{0},c_{1}}^{(\alpha)}\vert 
\vert F_{N,\alpha}( \frac{v}{N})\vert \\
&= O\left( (N-l+k)^{\alpha-1} (\frac{l-k}{N})^\alpha k_{1} N^{\alpha-1} 
\int_{0}^1 \ln (1-t+\frac{\alpha+1}{N}) dt \right)\\
&=o(N^{2\alpha-1})
\end{align*}
Lastly we obtain, still with an Abel summation 
\begin{align*} 
& \sum_{u=0}^{k-k_{1}-1} \delta _{N-k+u,\theta_{0}}^{(\alpha)} 
\overline{\delta_{N-l+u,\theta_{0}}^{(\alpha)}}\\
&= \frac{\vert K_{\alpha,\theta_{0},c_{1}}\vert ^2}{\Gamma^2(\alpha)}
 \cos\left( (l-k) \theta_{0}\right) \sum_{u=0}^{k-k_{1}-1} 
(N-k+u)^{\alpha-1} (\frac{k-u} {N} )^\alpha (N-l+u)^{\alpha-1} (\frac{l-u} {N} )^\alpha \\
&+o(N^{2\alpha-1}).
\end{align*}
Merging this last equality with (\ref{DEUXIO}) we obtain  
\begin{equation} \label{TERTIO}
S_{2,\alpha} =\vert K_{\alpha,\theta_{0},c_{1}} \vert^2
 \cos\left( (l-k) \theta_{0}\right) \sum_{u=0}^{k}
{\tilde\delta _{N-k+u}}^{(\alpha)} 
\overline{{\tilde\delta_{N-l+u}}^{(\alpha)}}
+o\left(N^{2\alpha-1}\right).
\end{equation}
The equations (\ref{DEUXIO}) and (\ref{TERTIO})
and Lemma \ref{GS} provide Theorem \ref{TOEP1}
for the case $\frac{1}{2} >\alpha>0$. The uniformity is a direct consequence of Theorem \ref{COEF} 
and Lemmas \ref{INVERS3} and \ref{final}.
\section{Proof of Corollary \ref{DEMI} and \ref{TOEP2}}
\begin{lemma}
For $\theta_{0} \in ]0,\pi[$ and 
$\alpha\in ]0, \frac{1}{2}[$ we have 
$$ \Vert T_{N} \left( 2(\cos \theta - \cos \theta_{0}) c_{1}\right) 
- T_{N} \left( 2^{2\alpha} (\cos \theta-\cos \theta_{0})^{2\alpha} c_{1}\right) \Vert 
\le K (\frac{1}{2}-\alpha) N$$
where $K$ is a constant no depending from $N$.
\end{lemma}
\begin{proof}{}
By the main value Theorem we have 
$$ \vert  2^{2\alpha} (\cos \theta-\cos \theta_{0})^{2\alpha} -2(\cos \theta - \cos \theta_{0})  \vert 
\le 4 (1-2\alpha)  2^{c_{\alpha}(\theta)} (\cos \theta-\cos \theta_{0})
^{c_{\alpha}(\theta)} \vert $$
with $ 0 <c_{\alpha}(\theta)<1-2\alpha$. Hence 
the function $\psi_\alpha \mapsto \theta \mapsto2^{c_{\alpha}(\theta)} (\cos \theta-\cos \theta_{0})^{c_{\alpha}(\theta)}
c_{1}(\theta)$ is in $L^1 (\mathbb T)$ .
For all integer $k$, $0\le k \le N$ we consider the integral \\
$ I_{k}=\int _{0}^{2 \pi} \left( 2^{2\alpha} (\cos \theta-\cos \theta_{0})^{2\alpha} -2(\cos \theta - \cos \theta_{0}) \right) c_{1}(\theta) e^{-ik\theta} d\theta.$\\
Assume $\frac{1}{2}-\alpha\rightarrow 0$ and put $\epsilon$, $0<\epsilon< 1-2\alpha $, for $\alpha$ sufficiently closed from $\frac{1}{2}$. Put $I_{k}=I_{k,1}+I_{k,2}+I_{k,3}$ with 
\begin{align*}
I_{k,1}&= \int _{0}^{\theta_0 -\epsilon} \left( 2^{2\alpha} (\cos \theta-\cos \theta_{0})^{2\alpha} -2(\cos \theta - \cos \theta_{0}) \right) c_{1}(\theta) e^{-ik\theta} d\theta,\\
I_{k,2}&= \int _{\theta_0-\epsilon}^{\theta_0+\epsilon} \left( 2^{2\alpha} (\cos \theta-\cos \theta_{0})^{2\alpha} 
-2(\cos \theta - \cos \theta_{0}) \right)  c_{1}(\theta) e^{-ik\theta} d\theta,\\
I_{k,3}&= \int _{\theta_0+\epsilon}^{2\pi}  \left( 2^{2\alpha} (\cos \theta-\cos \theta_{0})^{2\alpha} -2(\cos \theta - \cos \theta_{0}) \right)  c_{1}(\theta) e^{-ik\theta} d\theta.
\end{align*}
It is easy to see that  $\vert I_{k,1}\vert$ and 
$\vert  I_{k,3}\vert $ are bounded by $M (1-2\alpha)$ with $M$ is a positive real no depending from
$k$ or $N$. Easily $\vert I_{k,2}\vert \le 1-2\alpha
\Vert \psi_\alpha \Vert_1$. Hence
$\vert I_{k}\vert \le M_{2} (1-2\alpha)$ where $M_{2}$ is a positive real no depending from
$k$ or $N$.\\
In the other hand it is well known that for a $N \times N$ matrix $A$ we have 
$$ \Vert A \Vert \le \left( \sum_{i=1}^N \sum_{j=1}^N A^2_{i,j}\right)^{\frac{1}{2}}.$$
This last result achieves the proof.
\end{proof}
\begin{lemma}
Let $\delta>0$ a fixed real. For  $\alpha<\frac{1}{2}$ 
such that $\frac{1}{2} - \alpha$ sufficiently near of
zero we have for all integer $k$, $0\le k\le N$
$$ \Vert T_N^{-1} \left( \vert \chi - \chi_0\vert 
\vert \chi - \bar \chi_0\vert c_1\right) (\chi^k) 
- T_N^{-1} \left( \vert \chi - \chi_0\vert^{2\alpha} 
\vert \chi - \bar \chi_0\vert^{2\alpha} c_1\right) (\chi^k) \le o(N^{-\delta}).$$
\end{lemma}
\begin{proof}{}
Let us denote by 
$T_{1/2,N}$ the matrix $T_N^{-1} \left( \vert \chi - \chi_0\vert \vert \chi - \bar \chi_0\vert c_1\right)$
and by $T_{\alpha,N}$ the matrix 
$T_N^{-1} \left( \vert \chi - \chi_0\vert^{2\alpha} 
\vert \chi - \bar \chi_0\vert^{2\alpha} c_1\right) $.
Obviously
$$ T_{1/2,N} = T_{\alpha,N} 
\left( Id + T_{\alpha,N}^{-1} 
\left( T_{1/2,N} -T_{\alpha,N}\right)\right).$$
 Corollary \ref{DEUX} and Lemma \ref{GS} imply 
 the existence of a positve real $C$ such that 
 for all integers \\ $k, l$, $0\le k,l\le N$, we have 
 \begin{equation} \label{MAJOR8}
 \left( T_{\alpha,N} \right)_{k,l}\le N^C.
 \end{equation}
Since 
$$ \Vert T_{\alpha,N} \left( T_{1/2,N} -T_{\alpha,N}
\right)\Vert \le \Vert T^{-1}_{\alpha,N} \Vert
\Vert T_{1/2,N} - T_{\alpha,N}\Vert $$
the previous Lemma implies 
\begin{equation}\label{TOTO}
 \Vert T^{-1} _{\alpha,N}\Vert  
\Vert T_{1/2,N}-T_{\alpha,N} \Vert \le C' 
(\frac{1}{2} - \alpha) N^{C+2}.
\end{equation}
Put $\alpha= \frac{1}{2} - o(N^{-(C+3)})$.
From (\ref{TOTO}) 
the matrix 
$ \left(Id+ T_{\alpha,N}^{-1}
 (T_{1/2,N} -T_{\alpha,N})\right) ^{-1}$ is defined 
 and 
$T_{1/2,N} ^{-1} =   \left(Id+ T_{\alpha,N}^{-1}
 (T_{1/2,N} -T_{\alpha,N})\right) ^{-1} 
 T_{\alpha,N}^{-1}$. 
 Then we can write, for all integer $k$, $0\le k\le N$ 
 \begin{align*}
\Vert T_{1/2,N}^{-1} (\chi^k) - T_{\alpha,N}^{-1}(\chi^k)\Vert &\le  
\Bigl\Vert \left (\left( Id + T_{\alpha,N}^{-1} 
( T_{1/2,N} -T_{\alpha,N}) \right)^{-1} -Id\right)
T_{\alpha,N}^{-1} (\chi^k)\Bigr\Vert \\
&\le \frac{ \Vert T_{\alpha,N}^{-1} 
( T_{1/2,N} -T_{\alpha,N}) \Vert \Vert T_{\alpha,N}^{-1} (\chi^k)\Vert } {1- \Vert T_{\alpha,N}^{-1} 
(T_{1/2,N}-T_{\alpha,N})\Vert}.
\end{align*}
As for the equation (\ref{MAJOR8}) we have obviously a constant $J$ no depending from $k$ or $N$ 
such that \\
$\Vert T_{\alpha,N} ^{-1} (\chi^k) \vert \le O(N^J)$.
If $\alpha= \frac{1}{2} +o(N^{-(C+J+3+\delta)})$
we have 
$$ \Vert T_N^{-1} \left( \vert \chi - \chi_0\vert 
\vert \chi - \bar \chi_0\vert c_1\right) (\chi^k) 
- T_N^{-1} \left( \vert \chi - \chi_0\vert^{2\alpha} 
\vert \chi - \bar \chi_0\vert^{2\alpha} c_1\right) (\chi^k) \le o( N^{-\delta})$$
 \end{proof}
\section{Appendix}
\subsection{Estimation of a trigonometric sum}
\begin{lemma}\label{APPENDIX1}
Let $M_{0},M_{1}$ two integers with $0<M_{0}<M_{1}$, $\chi\neq 1$  and $f$ a function in 
$\mathcal C^1\left(]M_{0},M_{1}[\right)$  such that for all $t\in ]M_{0},M_{1}[$ $f(t) = O(t^{\beta})$ 
and $f'(t) = O(t^{\beta-1})$. Then 
$$\Bigl \vert\sum_{u=M_{0}}^{M_{1}} f(u) \chi^u  \Bigr \vert = \Bigl \lbrace \begin{array}{l}
 O(M_{1}^\beta) \quad \mathrm{if} \quad \beta>0\\
O(M_{0}^\beta) \quad \mathrm{if} \quad \beta<0.
\end{array}$$
\end{lemma}
\begin{proof}{}
With an Abel summation we obtain, if $\sigma_{u}= 1 + \cdots + \chi^u$,  
$$
\sum_{u=M_{0}}^{M_{1}} f(u) \chi^u  = \sum_{u=M_{0}}^{M_{1}-1} \left( f(u+1) -f(u)\right) 
\sigma_{u} +f(M_{1}) \sigma_{M_{1}} + f(M_{0}) \sigma_{M_{0}-1} 
$$
and
\begin{eqnarray*}
\sum_{u=M_{0}}^{M_{1}-1} \left( f(u+1) -f(u)\right) 
\sigma_{u} 
&=& \left(f(M_{0}) +f(M_{1}) \right)\left( \frac{1}{1-\chi}\right) - 
\sum_{u=M_{0}}^{M_{1}-1} \left( f(u+1) -f(u) \right) \frac{\chi^{u+1}}{1-\chi}\\
& =& \sum_{u=M_{0}}^{M_{1}-1} f'(c_{u}) \frac{\chi^{u+1}}{1-\chi}+ \left(f(M_{0}) +f(M_{1}) \right)\left( \frac{1}{1-\chi}\right)
\end{eqnarray*}
with $c_{u}\in ]u,u+1[$. 
We have 
$$\Bigl \vert   \sum_{u=M_{0}}^{M_{1}-1} f'(c_{u}) \frac{\chi^{u}}{1-\chi}\Bigr\vert \le 
O\left( \sum_{u=M_{0}}^{M_{1}-1} u^{\beta-1} \right)$$
hence 
$$\Bigl \vert\sum_{u=M_{0}}^{M_{1}} f(u) \chi^u  \Bigr \vert = \Bigl \lbrace
 \begin{array}{l}
 O(M_{1}^\beta) \quad \mathrm{if} \quad \beta>0\\
O(M_{0}^\beta) \quad \mathrm{if} \quad \beta<0.
\end{array}$$
\end{proof}

 \bibliography{Toeplitzdeux}

\end{document}